\documentclass[11pt,reqno]{amsart}

\usepackage[colorlinks,linkcolor=blue,anchorcolor=blue,citecolor=red,CJKbookmarks=True]{hyperref}
\usepackage{lineno,hyperref}
\modulolinenumbers[5]
\usepackage{amsmath}
\usepackage{amssymb}
\usepackage{mathrsfs}
\usepackage{amsthm}
\usepackage{bm}
\usepackage{graphicx}
\usepackage{booktabs}
\usepackage{float}
\usepackage{subfig}
\usepackage{geometry}
\usepackage{bbm}
\usepackage{color}
\newtheorem{Def}{Definition}[section]
\newtheorem{lem}[Def]{Lemma}
\newtheorem{cor}[Def]{Corollary}
\newtheorem{tho}[Def]{Theorem}
\newtheorem{pro}[Def]{Proposition}

\newtheorem{rems}[Def]{Remarks}
\newtheorem{exam}[Def]{Example}

\newcommand{\la}{\langle}
\newcommand{\ra}{\rangle}
\newcommand{\E}{\mathbb E}
\newcommand{\D}{\mathbb D}
\newcommand{\R}{\mathbb R}
\newcommand{\uH}{\mathbb H}
\newcommand{\ud}{\mathrm d}
\newcommand{\z}{\mathbf z}
\newcommand{\y}{\mathbf y}

\numberwithin{equation}{section}
\allowdisplaybreaks[4]

\begin{document}

\title[]{Convergence of Density Approximations 
for Stochastic Heat Equation}

\author{Chuchu Chen }
\address{1. LSEC, ICMSEC, \\
Academy of Mathematics and Systems Science,\\ Chinese Academy of Sciences,\\
Beijing, 100190, China\\ 
2. School of Mathematical Science,\\ University of Chinese Academy of Sciences,\\ Beijing, 100049, China}
\curraddr{}
\email{chenchuchu@lsec.cc.ac.cn}
\thanks{}

\author{Jianbo Cui }
\address{
 School of Mathematics, Georgia Institute of Technology, Atlanta, GA 30332, USA}
\curraddr{}
\email{jianbocui@lsec.cc.ac.cn}
\thanks{The work is supported by National Natural Science Foundation of China grants 11971470, 11871068, 11926417.}

\author{Jialin Hong}
\address{1. LSEC, ICMSEC, \\
Academy of Mathematics and Systems Science,\\ Chinese Academy of Sciences,\\
Beijing, 100190, China\\ 
2. School of Mathematical Science,\\ University of Chinese Academy of Sciences,\\ Beijing, 100049, China}
\curraddr{}
\email{hjl@lsec.cc.ac.cn}
\thanks{}

\author{Derui Sheng}
\address{1. LSEC, ICMSEC, \\
Academy of Mathematics and Systems Science,\\ Chinese Academy of Sciences,\\
Beijing, 100190, China\\ 
2. School of Mathematical Science,\\ University of Chinese Academy of Sciences,\\ Beijing, 100049, China}
\curraddr{}
\email{sdr@lsec.cc.ac.cn (Corresponding author)}
\thanks{}


\begin{abstract}
This paper investigates the convergence of density approximations for stochastic heat equation in both    uniform convergence  topology and total variation distance. 
The convergence order of the densities in uniform convergence  topology is shown to be exactly $1/2$ in the nonlinear case and nearly $1$ in the linear case. This result implies that
the distributions of the approximations always converge to the distribution of the origin equation in total variation distance.
As far as we know, this is the first result on the convergence of  density approximations to the stochastic partial differential equation. 

\end{abstract}

\maketitle

\section{Introduction}\label{S1}

%

In this paper, we consider the  stochastic heat equation driven by space-time white noise:
 \begin{equation}\label{SHE}
\partial_t u(t, x)=\partial_{xx}u(t, x)+b(u(t, x))+\sigma\dot{W}(t, x),~(t, x) \in(0, T] \times[0,1]
\end{equation}
with initial value $u(0, x)=u_{0}(x),~x \in[0,1]$ and Neumann boundary condition $\partial_{x}u(t, 0)=\partial_{x}u(t, 1)=0,~t \in[0, T].$
Here, $T>0$ is a fixed number and $\sigma\neq0$ is a constant. Eq. \eqref{SHE} arising in many physical problems, characterizes the evolution of a scalar field in a space-time-dependent random medium. 
The choice of the white noise as random potential corresponds to considering those regimes with very rapid variations, the type of turbulent flows (see \cite{BC95}). The density function of the solution characterizes all relevant
probabilistic information. 
Concerning the density of $u(t,x)$, its existence, regularity and strictly positivity under suitable assumptions have been well studied (e.g. \cite{BP98,MD08,NQ09}).
If the coefficient $b$ in Eq. \eqref{SHE} is infinitely differentiable with bounded derivatives, then as a direct consequence of \cite{MD08},
 for any $0 \le x\le1,\, t > 0,$
$u(t,x)$ admits a smooth density, and as is shown in \cite{BP98},
for any $0 \le x_1 <\cdots< x_d \le1,\, t > 0,$ the law of $(u(t, x_1),\cdots,u(t, x_d))$ admits a strictly positive smooth density, which can be seen as a regularity result for the marginal distribution of $\mathcal{C}([0,1])$-valued random variable $u(t,\cdot)$. 
Moreover,
if $b$  is continuously differentiable with bounded derivative, \cite{NQ09} gives the lower and upper Gaussian bound for the density of $u(t,x)$. 

It is a challenge topic to obtain the density exactly or even approximately. However, for stochastic heat equations, even for stochastic partial differential equations, to the best of our knowledge, there are 
 few results concerning the approximation of the density of the origin equation. 
 The purpose of this paper is to develop a strategy to investigate the convergence of
 the density approximations of the exact solution to Eq. \eqref{SHE}  in a suitable topology, via a sequence of 
perturbed equations 
   \begin{equation}\label{SAEE}
  \partial_t u^{\delta}(t, x)=\partial_{xx}u^{\delta}(t, x)+b(u^{\delta}(\delta[t/\delta], x))+\sigma\dot{W}(t, x),~(t, x) \in(0, T] \times[0,1],
\end{equation}
where $\delta=T/N$, $N\in\mathbb N^{+}$
and $[\cdot]$ denotes the greatest-integer function.
 Different from Eq. \eqref{SHE}, the drift term $(t,x)\rightarrow b(u^{\delta}(\delta[t/\delta], x))$ in Eq. \eqref{SAEE} being a piecewise constant in the variable $t$ is a discontinuous function
and converges to the drift term of Eq.  \eqref{SHE} as $\delta$ tends to $0$ formally. Hence it is natural and important to study the existence and  convergence of the density of Eq. \eqref{SAEE}. 
 Our main results are the convergence of density in total variation distance and convergence order of density in uniform convergence topology.
 \begin{tho}\label{dstyt}
Assume that $b\in\mathcal C_{\mathbf b}^\infty$, $\delta\in\left(0,\frac{T}{12}\wedge\frac{\log\frac{3}{2}}{4|b|_1}\right)$. Then there exists some constant $C=C(T,b,\sigma,\|u_0\|_E)$ such that for any $x\in(0,1)$,
 \begin{equation}\label{error1}
\|q_{T,x}^\delta-q_{T,x}\|_{L^\infty(\R)}\le C\delta^{\frac{1}{2}},
 \end{equation}
 where $q_{T,x}^\delta$ and $q_{T,x}$ are the densities of $u^\delta(T,x)$ and $u(T,x)$, respectively.
\end{tho}
\noindent In the particular case that $b$ is affine, the above convergence order $1/2$ of density can be improved to $1-\epsilon$ with some sufficient small $\epsilon>0$, which coincides with the strong convergence order $1-\epsilon$  in \cite{AP08}.  As far as we know, this is the first result on the  convergence of density approximations to the stochastic partial differential equation. 
Combining the uniformly 
boundedness of  $q_{T,x}^\delta$ in $L^1(\R)$ and Theorem \ref{dstyt}, it is concluded that $q_{T,x}^\delta$ converges to $q_{T,x}$ in $L^1(\R)$. 
This implies that
the distribution of $u^\delta(T,x)$  converges to the distribution of $u(T,x)$ in total variation distance.

 Our strategy to prove Theorem \ref{dstyt} is based on
the weak convergence analysis in the following sense
\begin{align}\label{WeaK}
|\mathbb E[f(u^\delta(T,x))-f(u(T,x))]|\le C\delta ^{\mu},
\end{align}
where $C$ is independent of $f$ and $\mu>0$.
One key ingredient for the test function-independent weak convergence analysis is the 
application of the Malliavin integration by parts formula, whose prerequisite is the above uniform non-degeneracy of $u^\delta(T,x)$. 
It is known that this non-degeneracy condition (see Definition \ref{Def1}) is exactly the condition of applying Bouleau--Hirsch’s criterion (see e.g. \cite[Theorem 2.1.4]{DN06}) to establish the existence and smoothness of the corresponding density. The major obstacle of this non-degeneracy lies in establishing the negative moments of the determinant of the corresponding Malliavin covariance matrix, which is overcome by proving a discrete version of comparison principle. 
Another difficulty is that
the moments of the Gateaux derivatives, as well as the Malliavin derivatives, of both $u(T,x)$ and $u^\delta(T,x)$ are dominated by the multiples of the corresponding Green function associated to  Neumann boundary condition, instead of being bounded by a constant. Based on the technical estimates on
 the Green function,
we remove the infinitesimal factor in the 
 weak convergence order of the numerical scheme in the literature (see \cite{DP09}) and prove that the weak convergence order in \eqref{WeaK}  is $1/2$. 
  In the particular case that  $b$ in Eq. \eqref{SHE} is affine, the convergence  order is improved to $1-\epsilon$ for some sufficient small $\epsilon>0$.

The paper is structured as follows. In Section \ref{S2}, Malliavin calculus associated to a white noise and the properties of the Green function are introduced briefly. Several auxiliary results concerning the regularity estimates of densities and derivatives of solutions are analyzed in Section \ref{S3}. Then we present the weak convergence analysis via Malliavin calculus in Section \ref{S4}. Finally, Section \ref{S5} is devoted to the convergence of density approximations of  Eq. \eqref{SAEE}.


\section{Preliminaries}\label{S2}

 Denote by $E$ the Banach space $\mathcal C([0,1])$ endowed with the norm $\|h\|_E=\sup_{x\in[0,1]}|h(x)|$and  by $\mathcal C_\mathbf{p}^\infty$ the set of all infinitely differentiable functions with polynomial growth from $\R$ to $\R$.
 Let $\mathcal C_\mathbf{b}^k$ be the set of all $k$ times continuous differentiable functions with bounded derivatives from $\R$ to $\R$ and $\mathcal C_\mathbf b^\infty:=\bigcap_{k\ge1}\mathcal C_\mathbf b^k$.
 For $b\in\mathcal \mathcal C_\mathbf{b}^k$, denote $|b|_i:=\sup_{x\in\R}|b^{(i)}(x)|$, $\forall\,i\in\{1,\cdots,k\}$. 
We denote  by  $\bm{\delta}_\z$ the Dirac delta function concentrated at $\z\in\R$ and by  $\mathbf x\wedge \mathbf y=\min\{\mathbf x,\mathbf y\}$, $\forall\,\mathbf x,\mathbf y\in\R$. Throughout this article, we use $C$ to denote a generic constant that may change from one place to another and depend on several parameters but never on the perturbation parameter $\delta$. When required, we will explicitly write $C(T,\sigma, \cdots )$ to emphasize the dependence of the constant $C$ upon the parameters $T,\sigma,\ldots$ In what follows, we adopt the conventions that a sum over an empty set is zero and   that $\frac{a}{0}=\infty$ provided $a>0$ is a constant.

In this section, we present some preliminaries, including some basic elements from Malliavin calculus associated to a white noise  
and several basic properties of the Green function associated to Eq. \eqref{SHE}.
 Let $\{W(t,x)\}_{\{(t,x)\in[0,T]\times[0,1]\}}$ be a Brownian sheet on $[0,T]\times[0,1]$, defined in a complete probability space $(\Omega,\mathscr{F},\mathbb{P})$. For $0\le t\le T$, let $\mathscr{F}_t$ be the $\sigma$-field generated by the random variables $\{W(s,x)\}_{\{(s,x)\in[0,t]\times[0,1]\}}$ and the $\mathbb{P}$-null sets. 
 In the context of Malliavin calculus, the isonormal Gaussian family $\{W(h),h\in\uH\}$ corresponding to $\uH:=L^2([0,T]\times[0,1])$ is  given by the Wiener integral
$$W(h)=\int_{0}^T\int_{0}^{1}h(s,y)W(ds,dy).$$
We are interested in Eq. \eqref{SHE} and always assume that $u_0\in E$ is deterministic. If the coefficient $b:\mathbb{R}\rightarrow\mathbb{R}$ satisfies the global Lipschitz condition, 
the rigorous meaning of Eq. \eqref{SHE} is given by means of (see e.g. \cite{WJ86}):
\begin{align}\label{exact sol}
u(t,x)=&\int_0^1G_{t}\left(x,y\right)u_0(y)\ud y+\int_{0}^{t}\int_{0}^{1}G_{t-s}(x,y)b(u(s,y))\ud y\ud s +\int_{0}^{t}\int_{0}^{1}G_{t-s}(x,y)\sigma W(\ud s,\ud y),
\end{align}
where $G_t(x,y),\,(t,x,y)\in\mathbb{R}_+\times(0,1)^2$, is the Green function associated to the stochastic heat equation on $[0,1]$ with Neumann boundary condition. 
Similarly, the mild solution
of Eq. \eqref{SAEE} is given by
\begin{align}\nonumber
u^\delta(t,x)=&\int_0^1G_{t}\left(x,y\right)u_0(y)\ud y+\int_{0}^{t}\int_{0}^{1}G_{t-s}(x,y)b(u^{\delta}(\delta[s/\delta], y))\ud y\ud s \\\label{perturbed sol}
&+\int_{0}^{t}\int_{0}^{1}G_{t-s}(x,y)\sigma W(\ud s,\ud y).
\end{align}
We would like to mention that the mild solution given by \eqref{perturbed sol} to Eq. \eqref{SAEE} corresponds to the accelerate exponential Euler scheme in numerical analysis (see e.g. \cite{AP08}). By denoting $t_i=i\delta$, we have
\begin{align*}
u^\delta(t_{i+1},x)=&\int_{0}^{1}G_{\delta}(x,y)u^\delta(t_{i},y)\ud y+\int_{t_i}^{t_{i+1}}\int_{0}^{1}G_{t_{i+1}-s}(x,y)b\left(u^\delta(t_{i},y)\right)\ud y\ud s \\
&+\int_{t_i}^{t_{i+1}}\int_{0}^{1}G_{t_{i+1}-s}(x,y)\sigma W(\ud s,\ud y).
\end{align*}

\subsection{Green function}
The
explicit formula of the Green function $G$ in \eqref{exact sol} is well-known,
 \begin{equation}\label{GD}
 G_{t}(x, y)=\frac{1}{\sqrt{4 \pi t}} \sum_{n=-\infty}^{+\infty}\left(e^{-\frac{(x-y-2 n)^{2}}{4 t}}+e^{-\frac{(x+y-2 n)^{2}}{4 t}}\right).
 \end{equation} 
Denote by $P_t(x,y)=\frac{1}{\sqrt{4\pi t}}\exp\left(-\frac{(x-y)^{2}}{4t}\right)$ the heat kernel on $\R$.
Hereafter, the following facts will be used frequently (\cite[Appendix]{BP98}):
\begin{itemize}
\item[(1)] For any $(t,x,y)\in\R_+\times(0,1)^2$, $G_t(x,y)>0$ and $\int_{0}^{1} G_{t}(x, y) \ud y=1$.
\item[(2)] Semigroup property: $\int_{0}^{1} G_{t}(x, y) G_{s}(y, z) \ud y=G_{s+t}(x, z)$,$\forall\,s,t\in\R_+,\,x,z\in(0,1)$.
\item[(3)] There exists a constant $K$ depending on $T$ such that for $(t,x,y)\in(0,T]\times(0,1)^2$,
\begin{equation}\label{GG0}
 \frac{1}{K}P_t(x,y) \le G_t(x,y)\le KP_t(x,y).
\end{equation}
\end{itemize}
For any $t>0$ and $x,y\in\R$, it is fairly understood that $P^2_t(x,y)=\sqrt{1/{8\pi t}}P_{t/2}(x,y)$ and 
$P_{s}(x,y)\le \sqrt{t/s}P_{t}(x,y)$ provided $0<s<t\le T$.
The explicit formula for $G_t(x,y)$ is complicated, whose estimation will be converted into the estimation of $P_t(x,y)$ in view of \eqref{GG0}. For instance, there exists $C=C(T)$ such that  $\forall\, x,\,y\in(0,1)$, $0<s<t\le T$,
\begin{equation}\label{Gts0}
G^2_{t}(x,y)\le\frac{C}{\sqrt{t}}G_{\frac{t}{2}}(x,y)
\end{equation}
and 
\begin{equation}\label{Gts}
G_{s}(x,y)\le C\sqrt{t/s}G_{t}(x,y). 
\end{equation}
\begin{lem}\label{Green}
For any $\nu\in(\frac{1}{3},1),$ there is $C=C(T,\nu)$ such that for any $0<s<t\le T$,
\begin{equation*}
\max\left(\int_0^1|G_t(x,y)-G_s(x,y)|\ud x,\int_0^1|G_t(x,y)-G_s(x,y)|\ud y\right)\le Cs^{-\nu}(t-s)^\nu.
\end{equation*}
\end{lem}
\begin{proof}
Similar to \cite[Corollary 3.4]{WJ86}, the series expansion in \eqref{GD} shows that
\begin{equation}\label{GPH}
G_t(x,y)=P_t(x,y)+H_t(x,y)
\end{equation}
with $H_t(x,y)\in C^\infty([0,T]\times[0,1]^2)$. From \cite[Corollary 2.2]{MR19}, we have for any $\nu\in(\frac{1}{3},1)$,
\begin{equation*}
\max\left(\int_\R|P_t(x,y)-P_s(x,y)|\ud x,\int_\R|P_t(x,y)-P_s(x,y)|\ud y\right)\le Cs^{-\nu}(t-s)^\nu.
\end{equation*}
Finally, the proof is completed by the facts that $H_t(x,y)\in C^\infty([0,T]\times[0,1]^2)$ and
\begin{align*}
|G_t(x,y)-G_s(x,y)|&\le |P_t(x,y)-P_s(x,y)|+|H_t(x,y)-H_s(x,y)|.
\end{align*}
\end{proof}

When considering Dirichlet boundary condition, all the results in the paper hold as well with minor modification because the Green function $\mathbf{G}$ corresponding to the Dirichlet boundary condition on $[0,1]$ is 
  \begin{equation}\label{Diri}
\mathbf{G}_{t}(x, y)=\frac{1}{\sqrt{4 \pi t}} \sum_{n=-\infty}^{+\infty}\left(e^{-\frac{(x-y-2 n)^{2}}{4 t}}-e^{-\frac{(x+y-2 n)^{2}}{4 t}}\right),
 \end{equation} 
whose property is very similar to $G$. For more information on the properties of $\mathbf{G}$, the reader is referred to \cite[Lemma 7]{BM90}. The following two-parameter Gronwall lemma is essential in the moment estimates in section \ref{S3}, whose proof is given in Appendix.
\begin{lem}\label{GW}
Let $g_{s,y}(t,x)\ge0$ satisfy
\begin{equation*}
g_{s,y}(t,x)\le CG_{t-s}(x,y)+C\int_{s}^{t}\int_0^1 G_{t-r_1}{(x,z_1)}g_{s,y}(r_1,z_1)\ud z_1\ud r_1,
\end{equation*}
for some constant $C>0$ and all $0<s<t\le T$ and $x,y\in(0,1)$. Then for some $C=C(T)$,
\begin{equation*}
g_{s,y}(t,x)\le CG_{t-s}(x,y).
\end{equation*}
\end{lem}

\subsection{Malliavin calculus associated to white noise}
 We denote by
$\mathcal{S}$ the class of smooth $\R$-valued random variables such that $\mathbf F\in\mathcal{S}$ has the form 
$\mathbf F=\mathbf f(W(h_1),\ldots,W(h_n))$,
where $\mathbf f$ belongs to $\mathcal C_\mathbf p^\infty(\mathbb{R}^n)$, $h_i\in \uH,\, i=1,\ldots,n, \,n\ge 1$. Here, $\mathcal C_\mathbf p^\infty(\mathbb{R}^n)$ is the set of all infinitely continuously differentiable functions $\mathbf f : \R^n\rightarrow \R$ such that $\mathbf f$ and all of its partial derivatives have polynomial growth.
Then for any $p\ge 1$ and integer $k\ge1$, we denote by $\mathbb{D}^{k,p}$ the completion of $\mathcal{S}$ with respect to the norm
\begin{equation*}
\|\mathbf F\|_{k,p}=\left(\mathbb{E}\left[|\mathbf F|^p+\sum_{j=1}^{k}\|D^j\mathbf F\|_{\mathbb{H}^{\bigotimes j}}^p\right]\right)^{\frac{1}{p}},
\end{equation*}
where $D$ is the Malliavin derivative operator. In particular, for $p\ge1$, we simply write $\|\mathbf F\|_p$ as an abbreviation for $\|\mathbf F\|_{0,p}$.
Define
$$L^{\infty-}(\Omega):=\bigcap_{p\ge 1} L^p(\Omega),\,\mathbb{D}^{k,\infty}:=\bigcap_{p\ge 1} \mathbb{D}^{k,p},\,\mathbb{D}^\infty:=\bigcap_{k\ge 1} \mathbb{D}^{k,\infty}$$
to be topological projective limits. As in the Schwartz theory of distributions, $\mathbb D^{-k,p}$ is the topological dual of the Banach space $\mathbb D^{k,p^\prime}$ with $1/p+1/p^\prime=1$ and
$\mathbb{D}^{-\infty}=\bigcup_{p\ge 1}\bigcup_{k\ge 1}\mathbb{D}^{-k,p}$ is the space of generalized Wiener functionals.  The natural coupling of $\mathbf F\in\mathbb{D}^{k,p}$ and $\Phi\in\mathbb{D}^{-k,q}$ with $1/p+1/q=1$ or that of $\mathbf F\in\mathbb{D}^{\infty}$ and $\Phi\in\mathbb{D}^{-\infty}$ is denoted by $\mathbb{E}[\mathbf F\cdot\Phi]$. 

\begin{Def}\label{Def1}
A random vector $\mathbf F=(\mathbf F_1,\mathbf F_2,\cdots,\mathbf F_m)$ whose components are in $\mathbb D^\infty$ is non-degenerate if the Malliavin covariance matrix $\Gamma_\mathbf F:=(\langle D\mathbf F_i,D\mathbf F_j\rangle_{\mathbb{H}})_{1\le i,j\le m}$ is invertible a.s. and $(\det \Gamma_\mathbf F)^{-1}\in L^{\infty-}(\Omega).$

%
\end{Def}
In the special case $m=1$, we still call $\Gamma_\mathbf F:=\|D\mathbf F\|^2_\mathbb{H}$ the Malliavin covariance matrix of $\mathbf F$, although $\Gamma_\mathbf F$ is actually a scalar variable.

\section{Technical estimates}\label{S3}
The classical weak convergence analysis of stochastic partial differential equations has been researched during the past two decades (see e.g. \cite{CB20,CHW19,DP09,HW19} and references therein), where the test function $\phi$ requires to have boundedness derivatives up to some degree, and the weak convergence order relies on the regularity of $\phi$. However, this kind of weak convergence for approximations is equivalent to the weak convergence of the associated distributions and is not sufficient to derive the convergence of densities. 
 By \cite[Lemma 2.1.7]{DN06}, the probability of the law of $\mathbf F$ at $\z\in\R$ can be determined by the generalized expectation $\E[\bm{\delta}_{ \z}(\mathbf F)]$, provided $\mathbf F$ is a non-degenerate random variable. For any fixed $ \z\in\R$, $\zeta>0$, we define the following mapping
 \begin{equation}\label{Gaussian}
 \y\mapsto g_{\zeta}(\y-\z)=\frac{1}{\sqrt{2\pi \zeta}}e^{-\frac{|\y- \z|^2}{2\zeta}}.
 \end{equation}
  It is well known that $g_{n^{-1}}(\cdot- \z)\rightarrow \bm{\delta}_{ \z}(\cdot)$ as $n$ tends to $\infty$ in the distribution sense and is natural to consider the error between $\E\left[g_{n^{-1}}(u^\delta(T,x)- \mathbf z)\right]$ and $\E\left[g_{n^{-1}}(u(T,x)- \mathbf z)\right]$.
Therefore, an alternative space that test function $f$ lives in to derive the convergence of density of  Eq. \eqref{SAEE} is 
 $$\Psi:=\{f:\R\rightarrow\R|f\in\mathcal C_\mathbf{p}^\infty,\exists \,F:\R\rightarrow\R \,\text{such that}\, 0\le F\le 1\,\text{and}\,F^\prime=f\},$$
since $\{g_{n^{-1}}(\cdot- \z)\}_{n\ge1,\,\z\in\R}$ is an element of $\Psi$. In this section, we prove some technical results in preparation for the following test-function independent weak convergence analysis result. 
 \begin{tho}\label{err} 
Let $b\in\mathcal C_{\mathbf b}^\infty$, $\delta \in\left(0,\frac{T}{12}\wedge\frac{\log\frac{3}{2}}{4|b|_1}\right)$. Then there exists some positive constant $C=C(T,b,\sigma,\|u_0\|_E)$ such that for any $x\in(0,1)$ and $f\in\Psi$, it holds that
 \begin{equation}\label{order}
\left|\E[f(u^\delta(T,x))]- \E\left[f(u(T,x))\right]\right|\le C\delta^{\frac{1}{2}}.
 \end{equation}
 \end{tho}
 \subsection{Error decomposition}
In order to prove Theorem \eqref{err}, the following notations are introduced for simplicity.
 %
For $0\le s<t\le T,\,x\in(0,1)$ and $v: \Omega\rightarrow  E$ being $\mathscr F_s$-measurable, we denote by $\varphi_t^x(s,v)$ (resp. $\Phi_t^x(s,v)$) the exact flow of Eq. \eqref{SHE} (resp. Eq. \eqref{SAEE}). More precisely,
 \begin{align*}
\varphi_t^x(s,v)=&\int_0^1G_{t-s}\left(x,z\right)v(z)\ud z+\int_{s}^{t}\int_{0}^{1}G_{t-r}(x,z)b(\varphi_r^z(s,v))\ud r\ud z\\&+\int_{s}^{t}\int_{0}^{1}G_{t-r}(x,z)\sigma W(\ud r,\ud z),\\
\Phi_t^x(s,v)=&\int_{0}^{1}G_{t-s}(x,z)v(z)dz+\int_{s}^{t}\int_{0}^{1}G_{t-r}(x,z)b\left(\Phi_{\delta[r/\delta]}^z(s,v)\right)\ud r\ud z\\&+\int_{s}^{t}\int_{0}^{1}G_{t-r}(x,z)W(\ud r,\ud z).
\end{align*}
 The Gateaux derivative of $\varphi_t^x(s,\cdot)$ at $v\in E$ in the direction $h\in E$ is defined formally by
\begin{align*}
\la\mathcal{D}f(\varphi_t^x(s,v)),h\ra&=\frac{d}{d\epsilon}f(\varphi_t^x(s,v+\epsilon h))|_{\epsilon=0}=f^\prime(\varphi_t^x(s,v))\la\mathcal{D}\varphi_t^x(s,v),h\ra.
\end{align*}
 For $i\in\{1,\cdots,N\}$ and $\beta,\tau\in[0,1]$, we denote  
\begin{align}\label{Yi}
Y_i^\tau&:=\tau\Phi_{t_i}(t_{i-1},\Phi_{t_{i-1}}(0,u_0))+(1-\tau)\varphi_{t_i}(t_{i-1},\Phi_{t_{i-1}}(0,u_0)),\\\label{Zry}
Z_i^\beta(r,y)&:=\beta\varphi_r^y(t_{i-1},\Phi_{t_{i-1}}(0,u_0))+(1-\beta)\Phi^y_{t_{i-1}}(0,u_0),\,r\in(t_{i-1},t_i],\,y\in[0,1] .
\end{align}
Since $Y_i^\tau\in E$, a.s, for $y\in[0,1]$, we write
\begin{align*}
Y_i^\tau(y)&:=\tau\Phi^y_{t_i}(t_{i-1},\Phi_{t_{i-1}}(0,u_0))+(1-\tau)\varphi^y_{t_i}(t_{i-1},\Phi_{t_{i-1}}(0,u_0)).
\end{align*}
Using the above notations, the one-step error between Eq. \eqref{SHE} and Eq. \eqref{SAEE} is divided into 
\begin{align*}\nonumber
 &\varphi_{t_i}(t_{i-1},\Phi_{t_{i-1}}(0,u_0))-\Phi_{t_i}(t_{i-1},\Phi_{t_{i-1}}(0,u_0))\\
 =&\int^{t_i}_{t_{i-1}}\int_{0}^{1}G_{t_i-r}(\cdot,y)\left(b(\varphi_r^y(t_{i-1},\Phi_{t_{i-1}}(0,u_0)))-b(\Phi^y_{t_{i-1}}(0,u_0))\right)\ud y\ud r\\
 =&\int_0^1\int^{t_i}_{t_{i-1}}\int_{0}^{1}G_{t_i-r}(\cdot,y)b^\prime(Z_i^\beta(r,y))\left(\varphi_r^y(t_{i-1},\Phi_{t_{i-1}}(0,u_0))-\Phi^y_{t_{i-1}}(0,u_0)\right)\ud y\ud r\ud\beta\\
  =&\int_0^1\int^{t_i}_{t_{i-1}}\int_{0}^{1}G_{t_i-r}(\cdot,y)b^\prime(Z_i^\beta(r,y))\int_0^1\{G_r(y,\xi)-G_{t_{i-1}}(y,\xi)\}u_0(\xi)\ud\xi\ud y\ud r\ud\beta\\
  +&\int_0^1\int^{t_i}_{t_{i-1}}\int_{0}^{1}G_{t_i-r}(\cdot,y)b^\prime(Z_i^\beta(r,y))\int_0^{t_{i-1}}\int_0^1\{G_{r-\theta}(y,\xi)-G_{t_{i-1}-\theta}(y,\xi)\}\\
  &\qquad\qquad\qquad\qquad\qquad\qquad\qquad\qquad\qquad\qquad\qquad\qquad\qquad b\left(\Phi_{\lfloor\theta\rfloor}^\xi(0,u_0)\right)\ud\xi\ud\theta\ud y\ud r\ud\beta\\
  +&\int_0^1\int^{t_i}_{t_{i-1}}\int_{0}^{1}G_{t_i-r}(\cdot,y)b^\prime(Z_i^\beta(r,y))\int_0^{t_{i-1}}\int_0^1\{G_{r-\theta}(y,\xi)-G_{t_{i-1}-\theta}(y,\xi)\}\sigma W(\ud\theta,\ud\xi)\ud y\ud r\ud\beta\\
  +&\int_0^1\int^{t_i}_{t_{i-1}}\int_{0}^{1}G_{t_i-r}(\cdot,y)b^\prime(Z_i^\beta(r,y))\int_{t_{i-1}}^r\int_0^1G_{r-\theta}(y,\xi)b\left(\varphi_\theta^\xi(t_{i-1},\Phi_{t_{i-1}}(0,u_0)\right)\ud \xi\ud\theta\ud y\ud r\ud\beta\\
  +&\int_0^1\int^{t_i}_{t_{i-1}}\int_{0}^{1}G_{t_i-r}(\cdot,y)b^\prime(Z_i^\beta(r,y))\int_{t_{i-1}}^r\int_0^1G_{r-\theta}(y,\xi)\sigma W(\ud \theta,\ud\xi)\ud y\ud r\ud\beta=:\sum_{j=1}^5R_j^i,
\end{align*}
Supposing that $f\in\Psi$, we consider the telescoping sum
\begin{align}\nonumber\label{weak error}
&\E\left[f(\varphi_T^x(0,u_0))\right]-\E\left[f(\Phi_T^x(0,u_0))\right]\\\nonumber
=&\sum_{i=1}^N\E\left[f(\varphi_T^x(t_i,\varphi_{t_i}(t_{i-1},\Phi_{t_{i-1}}(0,u_0))))-f(\varphi_T^x(t_i,\Phi_{t_i}(t_{i-1},\Phi_{t_{i-1}}(0,u_0))))\right]\\
=&\sum_{i=1}^N\sum_{j=1}^5\E\left[\int_0^1\left\la\mathcal{D}f(\varphi_T^x(t_i,Y_i^\tau)),R_j^i\right\ra\ud\tau\right].
 \end{align} 
By chain rule, we have
\begin{align*}
\E[\mathcal I_i^j]:=\E\left[\int_0^1\left\la\mathcal{D}f(\varphi_T^x(t_i,Y_i^\tau)),R_j^i\right\ra\ud\tau\right]=\E\left[\int_0^1f^\prime(\varphi_T^x(t_i,Y_i^\tau))\left\la \mathcal{D}\varphi_T^x(t_i,Y_i^\tau),R_j^i\right\ra\ud\tau\right].
\end{align*}
  
The above error decomposition \eqref{weak error} is standard, however, the appearance of $f^\prime(\varphi_T^x(t_i,Y_i^\tau)$ in $\E[\mathcal I_i^j]$ and the requirement that $C$ is independent of $f\in\Psi$ imply that the classical estimates that $f^\prime(\varphi_T^x(t_i,Y_i^\tau)$ is bounded by $|f|_1$ is not apply to our case. The Malliavin integration by parts formula  (see e.g. \cite[Proposition 2.1.4]{DN06}) has been used to remove the dependence of $C$ upon $f$ in  the weak convergence of approximation for  stochastic ordinary differential equation (see e.g. \cite{BT96}). We applying this idea to the stochastic heat equation and obtain Lemma \ref{G11}, whose proof is based on the non-degeneracy of $\varphi_T^x(t_i,Y_i^\tau)$ and is given in Subsection \ref{S32}. To avoid ambiguity, we point out that $f^{(\alpha)}(\varphi_T^x(t_i,Y_i^\tau))$ denotes the composition of the $\alpha$-th derivative $f^{(\alpha)}$ of $f$ and the random variable $\varphi_T^x(t_i,Y_i^\tau)$. 
   \begin{lem}\label{G11}
Let $\alpha\in\mathbb{N}$, $b\in\mathcal C_{\mathbf b}^\infty$ and  $\delta \in\left(0,\frac{T}{12}\wedge\frac{\log\frac{3}{2}}{4|b|_1}\right)$. If $G_1\in\D^\infty$ and $f\in\Psi$, then for any $i\in\{1,\cdots,N\}$, $x\in(0,1)$ and $\tau\in[0,1]$,
  \begin{align*}
\left|\E\left[f^{(\alpha)}(\varphi_T^x(t_i,Y_i^\tau))G_1\right]\right|\le C\|G_1\|_{k,p}
\end{align*}
holds for some constant $C=C(\alpha,k,p,T,\sigma,b,\|u_0\|_E)$.
    \end{lem}

\subsection{Regularity of densities}\label{S32}

In this part, we study the non-degeneracy property of
$$\{\varphi_T^x(t_i,Y_i^\tau)\}_{\{x\in(0,1),\,i\in\{1,\cdots,N\},\,\tau\in[0,1]\}},$$ which 
indicates the existence and smoothness of its density. It is noteworthy that both the solutions $u(T,x)$ to Eq. \eqref{SHE}  and $u^\delta(T,x)$ to Eq. \eqref{SAEE} are special cases of $\varphi_T^x(t_i,Y_i^\tau)$ since $u(T,x)=\varphi_T^x(t_1,Y_1^0)$ and $u^\delta(T,x)=\varphi_T^x(t_N,Y_N^1)$.  For more general SPDEs, as well as those driven by multiplicative or more rough noises, we refer to \cite{CH20,SS05} and references therein for a fruitful results of  research on densities of their exact solutions. In particular, as a direct consequence of \cite{MD08}, for any $0 <x<1,$
$u(T,x)$ is non-degenerate and thereby admits a smooth density.

 \subsubsection{Negative moments}
 To start with, we give a uniform positive lower bound, independent of the sample $\omega$ and the perturbation parameter $\delta$,  of the Malliavin covariance matrix $\Gamma_{\varphi_T^x(t_i,Y_i^\tau)}$ by proving a discrete version comparison principle.
\begin{pro}\label{inverse}
Let $x\in(0,1)$, $i\in\{1,\cdots,N\}$ and $\tau\in[0,1]$, and  assume that $b\in\mathcal C_{\mathbf b}^1$. Then for any $\delta \in\left(0,\frac{T}{12}\wedge\frac{\log\frac{3}{2}}{4|b|_1}\right)$, the Malliavin covariance matrix $\Gamma_{\varphi_T^x(t_i,Y_i^\tau)}$ satisfies
\begin{equation*}
\Gamma_{\varphi_T^x(t_i,Y_i^\tau)}\ge c,
\end{equation*}
for some $c=c(T,|b|_1,\sigma)>0$.
\end{pro}
\begin{proof}
Without loss of generality, assume that $\sigma>0$.
By the Cauchy-Schwarz inequality, we infer that for $i\ge 1$,
\begin{align}\label{gm}
\Gamma_{\varphi_T^x(t_i,Y_i^\tau)}=&\int_0^T\int_0^1D_{\theta,\xi}\varphi_T^x(t_i,Y_i^\tau)^2\ud\xi\ud\theta
\ge\int_0^T\left(\int_0^1D_{\theta,\xi}\varphi_T^x(t_i,Y_i^\tau)\ud\xi\right)^2\ud\theta.
\end{align}
Denote $X(t,x;\theta):=\int_0^1D_{\theta,\xi}\varphi_t^x(t_i,Y_i^\tau)\ud\xi$. Recalling the definition of $Y_i^\tau$ in \eqref{Yi}, $X(t,x;\theta)$ depends on $\tau$ and we drop its explicit dependence for simplicity.
Observing that by the definition of $\varphi_T^x(t_i,Y_i^\tau)$ and the chain rule,
 \begin{align}\nonumber\label{gmX}
X(T,x;\theta)=&\int_0^1G_{T-t_i}(x,y)\int_0^1D_{\theta,\xi}Y_i^\tau(y)\ud\xi\ud y +\int_{t_i}^T\int_0^1G_{T-r}(x,y)b^\prime(\varphi_r^y(t_i,Y_i^\tau))X(r,y;\theta)\ud y\ud r\\
&+G_{T-\theta}(x,\xi)\sigma\mathbf{1}_{\{\theta\in(t_i,T]\}}.
\end{align}
 In order to dominate $X(T,x;\theta)$ from below, we require to estimate $X(T,x;\theta)$ in two cases $\theta>t_i$ and $\theta<t_i$. In the first case, $\theta>t_i$ implies $D_{\theta,\xi}Y_i^\tau(y)=0$ because $Y_i^\tau$ is $\mathcal F_{t_i}$-measurable
 and in the second case, $\theta<t_i$ implies $G_{T-\theta}(x,\xi)\sigma\mathbf{1}_{\{\theta\in(t_i,T]\}}=0$, but the estimation of $D_{\theta,\xi}Y_i^\tau(y)$ requires a more sophisticated treatment than the first case.

\textsf{Case 1:}
Let $\theta\in(t_i,T]$.
Then it follows from \eqref{gmX} that
\begin{align*}
\partial_tX(t,x;\theta)=\partial_{xx}X(t,x;\theta)+b^\prime(\varphi_t^x(t_i,Y_i^\tau))X(t,x;\theta),\,\theta<t\le T
\end{align*}
with initial condition $X(\theta,x;\theta)=\sigma,\,\forall\,x\in(0,1)$ and the Nuemann boundary condition. By comparison principle (\cite[Lemma 4]{MD08}) and the assumption $b^\prime(\varphi_t^x(t_i,Y_i^\tau))\ge-|b|_1 >-\infty$, we obtain that for any $\tau\in[0,1]$,
\begin{equation}\label{St1}
X(T,x;\theta)\ge e^{-|b|_1 (T-\theta)}\sigma.
\end{equation}

 \textsf{Case 2:}
 Let $\theta\in(0,t_i)$. By \eqref{gmX}, we begin with estimating $\int_0^1D_{\theta,\xi}Y_i^\tau(y)\ud y$, which  is equivalent to estimating $\int_0^1D_{\theta,\xi}\varphi_{t_i}^y(t_{i-1},\Phi_{t_{i-1}}(0,u_0))\ud\xi$ and $\int_0^1D_{\theta,\xi}\Phi^y_{t_i}(0,u_0)\ud\xi$. Therefore, we are in a position to estimate $\int_0^1D_{\theta,\xi}\Phi_{t_i}(0,u_0)\ud\xi$.
 We denote for brevity $M_i(\theta,y):=\int_0^1D_{\theta,\xi}\Phi^y_{t_i}(0,u_0)\ud\xi.$ 
  Then for any $\theta\in(t_{k},t_{k+1})$ with $0\le k\le i-1$, $M_i(\theta,y)$ satisfies the following recursive relation
\begin{align*}
M_i(\theta,y)=&\sum_{j=k+1}^{i-1}\int_{t_j}^{t_{j+1}}\int_{0}^{1}G_{t_i-r}(y,z)b^\prime(\Phi^z_{t_j}(0,u_0))M_j(\theta,z)\ud z\ud r+\sigma.
\end{align*}

To get a lower bound, we prove a discrete version of comparison principle.
Define a two-parameter sequence $\{A_i^k\}_{1\le k\le i\le N}$ by for any $i\in\{1,\cdots,N\}$, $A_i^i=0$, $A_i^{i-1}=\sigma$ and
\begin{align}\nonumber\label{Aik}
A_i^k=&\sum_{j=k+1}^{i-1}\int_{t_j}^{t_{j+1}}\int_{0}^{1}G_{t_i-r}(y,z)|b|_1 A_j^k\ud z\ud r+\sigma\\
=&\sum_{j=k+1}^{i-1}|b|_1\delta  A_j^k+\sigma,\,\forall\,1\le k\le i-2.
\end{align}
By an induction argument and the construction of $A_i^k$, we see that for any $\theta\in(t_k,t_{k+1})$ and $y\in(0,1)$,
$$|M_i(\theta,y)|\le A_i^k.$$
By definition, if $i_1-k_1=i_2-k_2$, then $A_{i_1}^{k_1}=A_{i_2}^{k_2}=:\mathcal A_{i_1-k_1}$. 
Rearranging \eqref{Aik}, we derive
\begin{align*}\mathcal A_{i-k}&=\sum_{j=k+1}^{i-1}|b|_1\delta  \mathcal A_{j-k}+\sigma=|b|_1\delta  \mathcal A_{i-1-k}+\sum_{j=k+1}^{i-2}|b|_1\delta  \mathcal A_{j-k}+\sigma\\
&=
(1+|b|_1\delta )\mathcal A_{i-1-k}=(1+|b|_1\delta )^{i-k-1}\sigma.\end{align*}
Therefore, if $|b|_1>0$, we have
\begin{align*}
M_i(\theta,y)&\ge\sigma-|b|_1 \delta (A_{i-1}^{k}+A_{i-2}^{k}+\cdots+A_{k+1}^{k})\\
&=\left\{2-(1+|b|_1\delta )^{i-k-1}\right\}\sigma\ge\frac{1}{2}\sigma,
\end{align*}
provided $1\le i-k-1\le\frac{\log\frac{3}{2}}{\log(1+|b|_1\delta )}$. 
 Notice that $0<\log(1+x)\le x,\,\forall\,x>0$. 
To summarize,
 for any $y\in(0,1)$ and $\theta\in(t_k,t_{k+1})$ with $\max\left\{0,i-1-\frac{\log\frac{3}{2}}{|b|_1\delta }\right\}\le k\le i-2$, it holds that
\begin{align}\label{M1}
M_i(\theta,y)\ge\frac{1}{2}\sigma.
\end{align}
Obviously, if $|b|_1=0$, i.e. $b^\prime\equiv0$, the desired positive lower bound \eqref{M1} for $M_i(\theta,y)$ is valid as well.

Noticing that
\begin{align*}
&\int_0^1D_{\theta,\xi}\varphi_{t_i}^y(t_{i-1},\Phi_{t_{i-1}}(0,u_0))\ud\xi=\int_0^1G_{\delta }(y,z)\int_0^1D_{\theta,\xi}\Phi^z_{t_{i-1}}(0,u_0)\ud\xi\ud z\\
&\qquad+\int_{t_{i-1}}^{t_i}\int_0^1G_{t_i-r}(y,z)b^\prime(\varphi_r^z(t_{i-1},\Phi_{t_{i-1}}(0,u_0)))\int_0^1D_{\theta,\xi}\varphi_r^z(t_{i-1},\Phi_{t_{i-1}}(0,u_0))\ud\xi\ud z\ud r,
\end{align*}
we aim to derive a low bound for $\int_0^1D_{\theta,\xi}\varphi_{t_i}^y(t_{i-1},\Phi_{t_{i-1}}(0,u_0))\ud\xi$ by the above low bound of $\int_0^1D_{\theta,\xi}\Phi^z_{t_{i-1}}(0,u_0)\ud\xi=M_{i-1}(\theta,z)$ and the comparison principle (\cite[Lemma 4]{MD08}). To be precise,
supposing that $\theta\in(t_k,t_{k+1})$ and $y\in(0,1)$ with $\max\left\{0,i-1-\frac{\log\frac{3}{2}}{|b|_1\delta }\right\}\le k\le i-3$ are arbitrarily fixed, then for any $z\in(0,1)$, it holds that $\int_0^1D_{\theta,\xi}\Phi^z_{t_{i-1}}(0,u_0)\ud\xi\ge\frac{1}{2}\sigma$,
which, together with the comparison principle indicates that
\begin{align}\label{M2}
\int_0^1D_{\theta,\xi}\varphi_{t_i}^y(t_{i-1},\Phi_{t_{i-1}}(0,u_0))\ud\xi\ge \frac{1}{2}e^{-|b|_1 \delta }\sigma.
\end{align}
Thus by \eqref{M1} and \eqref{M2}, we have that for any $\tau\in[0,1]$,
\begin{align}\nonumber\label{M3}
\int_0^1D_{\theta,\xi}Y_i^\tau(y)\ud\xi&=\tau\int_0^1D_{\theta,\xi}\varphi_{t_i}^y(t_{i-1},\Phi_{t_{i-1}}(0,u_0))\ud\xi+(1-\tau)\int_0^1D_{\theta,\xi}\Phi_{t_i}(0,u_0)\ud\xi\\
&\ge \frac{\tau}{2}e^{-|b|_1 \delta }\sigma+(1-\tau)\frac{1}{2}\sigma\ge\frac{1}{2}e^{-|b|_1 \delta }\sigma.
\end{align}
Now we turn to \eqref{gmX} and  estimate $X(T,x;\theta)$. Taking account of \eqref{M3} and applying
the comparison principle yield that 
\begin{align}\label{st3}
X(T,x;\theta)&\ge e^{-|b|_1 (T-t_{i})}\int_0^1D_{\theta,\xi}Y_i^\tau(y)\ud\xi\ge \frac{1}{2}e^{-|b|_1 (T-t_{i-1})}\sigma,
\end{align}
for any $\tau\in[0,1]$ and $\theta\in(t_k,t_{k+1})$ with $\max\left\{0,i-1-\frac{\log\frac{3}{2}}{|b|_1\delta }\right\}\le k\le i-3$.

So far, we have  dominated $X(T,x;\theta)$ from below when $\theta>t_i$ in
 \textsf{Case 1} and when $\theta\in(t_k,t_{k+1})$ with $\max\left\{0,i-1-\frac{\log\frac{3}{2}}{|b|_1\delta }\right\}\le k\le i-3$ in \textsf{Case 2}, based on which, we are going to give a lower bound estimate of $\Gamma_{\varphi_T^x(t_i,Y_i^\tau)}$ as follows. By  \eqref{gm} and \eqref{St1},
\begin{align*}
\Gamma_{\varphi_T^x(t_i,Y_i^\tau)}
\ge&\int_0^T\left|X(T,x;\theta)\right|^2\ud\theta\\
\ge&\sum_{k=0}^{i-3}\int_{t_k}^{t_{k+1}}\left|X(T,x;\theta)\right|^2\ud\theta+\int_{t_i}^Te^{-2|b|_1 (T-\theta)}\sigma^2\ud\theta.
\end{align*}
For $0\le i\le\frac{N}{2}+3$,
\begin{align*}
\Gamma_{\varphi_T^x(t_i,Y_i^\tau)}
\ge&\int_{\frac{3T}{4}}^Te^{-2|b|_1 (T-\theta)}\sigma^2\ud\theta=\frac{1-e^{-\frac{|b|_1 T}{2}}}{2|b|_1 }\sigma^2=:c_1,
\end{align*}
in view of $\delta \le\frac{T}{12}$.
For $\frac{N}{2}+3\le i\le N$, we have $T-t_{i-1}\le \frac{T}{2}$ and thus
\begin{align*}
\Gamma_{\varphi_T^x(t_i,Y_i^\tau)}
\ge&\sum_{k=\max\left\{0,i-1-\frac{\log\frac{3}{2}}{|b|_1\delta }\right\}}^{i-3}\int_{t_k}^{t_k+1}\left|X(T,x;\theta)\right|^2\ud\theta\\
\ge& \frac{1}{4}e^{-2|b|_1 T}\sigma^2\delta \min\left\{\frac{N}{2},\frac{\log\frac{3}{2}}{|b|_1\delta }-2\right\}\ge \frac{1}{8}e^{-2|b|_1 T}\sigma^2\min\left\{T,\frac{\log\frac{3}{2}}{|b|_1}\right\}=:c_2,
\end{align*}
thanks to \eqref{st3} and $\delta \le\frac{\log\frac{3}{2}}{4|b|_1}.$ Finally,
we finish the proof by choosing $c=\min\left\{c_1,c_2
\right\}$.

\end{proof}

\subsubsection{Integrability of Malliavin derivatives}

Next lemma states that $\varphi_t^x(t_i,Y_i^\tau)$ and its Malliavin derivatives of any order have bounded moments. We would like to mention that this property is still valid for stochastic heat equation driven by multiplicative noise with further assumptions (see e.g. \cite{BP98}). 
\begin{lem}\label{Ykp0}
Assume that $b\in\mathcal C_{\mathbf b}^\infty$. Then 
for any integers $k\ge0,\,p\ge1$, there exists $C=C(k,p,T,b,\sigma,\|u_0\|_E)$ such that for any $\tau\in[0,1],$
\begin{align}\label{Yiy}
&\sup_{i=1,\cdots,N}\sup_{y\in(0,1)}\|\Phi^y_{t_i}(0,u_0)\|_{k,p}+\sup_{i=1,\cdots,N}\sup_{y\in(0,1)}\|\varphi_{t_i}^y(t_{i-1},\Phi_{t_{i-1}}(0,u_0))\|_{k,p}\le C,\\\label{Ykp}
&\sup_{i=1,\cdots,N}\sup_{t\in[t_i,T],x\in(0,1)}\|\varphi_t^x(t_i,Y_i^\tau)\|_{k,p}\le C.
\end{align}
\end{lem}
\begin{proof}
 Notice that for any $F\in\D^{k,p}$, it holds that
\begin{align}\label{Fkp}
\|F\|_{k,p}^p=\|F\|_{k-1,p}^p+\|D^kF\|^p_{L^p(\Omega,\uH^{\otimes k})}.
\end{align}
To begin with, let $i\in\{1,\cdots,N\}$ be arbitrarily fixed. By definition,
\begin{align*}
\Phi^y_{t_i}(0,u_0))=&\int_{0}^{1}G_{t_i}(y,z)u_0(z)\ud z+\sum_{j=0}^{i-1}\int_{t_j}^{t_{j+1}}\int_{0}^{1}G_{t_i-r}(y,z)b(\Phi^z_{t_j}(0,u_0))\ud z\ud r\\
&+\int_{0}^{t_i}\int_{0}^{1}G_{t_i-r}(y,z)\sigma W(\ud r,\ud z).
\end{align*}
By $u_0\in E$ and the linear growth of $b$, we have
\begin{align*}
\sup_{y\in(0,1)}\|\Phi^y_{t_i}(0,u_0))\|_p\le&C(T,\|u_0\|_E)+\sum_{j=0}^{i-1}\int_{t_j}^{t_{j+1}}\sup_{y\in(0,1)}\int_{0}^{1}G_{t_i-r}(y,z)\sup_{z\in(0,1)}\|\Phi^z_{t_j}(0,u_0)\|_p\ud z\ud r\\
&+\sup_{y\in(0,1)}\left\|\int_{0}^{t_i}\int_{0}^{1}G_{t_i-r}(y,z)\sigma W(\ud r,\ud z)\right\|_p.
\end{align*}
Therefore, the Burkholder's inequality and the discrete Gronwall lemma produce
\begin{equation}\label{Phi1}
\sup_{y\in(0,1)}\|\Phi^y_{t_i}(0,u_0)\|_p\le C,\,\forall\,i=1,\cdots,N.
\end{equation}
Similarly, by the definition of $\varphi_{t_i}^y(t_{i-1},\Phi_{t_{i-1}}(0,u_0))$,
the linear growth of $b$, Burkholder's and Minkowskii's inequalities, we have
\begin{align*}
\|\varphi_{t_i}^y(t_{i-1},\Phi_{t_{i-1}}(0,u_0))\|_p\le&\int_0^1G_{\delta }(y,z)\|\Phi^z_{t_{i-1}}(0,u_0)\|_p\ud z+C\left(\int_{t_{i-1}}^{t_i}\int_0^1G^2_{t_i-r}(y,z) \ud z\ud r\right)^{\frac{1}{2}}\\
&+C\int_{t_{i-1}}^{t_i}\int_0^1G_{t_i-r}(y,z)(1+\|\varphi_r^z(t_{i-1},\Phi_{t_{i-1}}(0,u_0))\|_p)\ud z\ud r.
\end{align*}
Taking the supremum over $y\in(0,1)$ and taking account of \eqref{Phi1}, we obtain that
\begin{equation}\label{Phi2}
\sup_{y\in(0,1)}\|\varphi_{t_i}^y(t_{i-1},\Phi_{t_{i-1}}(0,u_0))\|_p\le C,\,\forall\,i=1,\cdots,N,
\end{equation}
which together with \eqref{Phi1} implies that \eqref{Yiy} holds for $k=0$.
Similar to the process of the proof of \eqref{Phi2}, it can be shown that \eqref{Ykp} holds for $k=0$ as well. 

By induction, we assume that \eqref{Yiy} and \eqref{Ykp} hold up to the index $k-1,\,k\ge1$. 
By utilizing Leibnitz's rule, it holds that
\begin{align*}
&\|D^k\varphi_t^x(t_i,Y_i^\tau)\|_{L^p(\Omega,\uH^{\otimes k})}\le\int_0^1G_{t-t_i}(x,y)\|D^kY_i^\tau(y)\|_{L^p(\Omega,\uH^{\otimes k})}\ud y\\
&+\sigma\left(\int_{t_i}^t\int_0^1G^2_{t-s}(x,y)\ud y\ud s \right)^{\frac{1}{2}}\mathbf{1}_{\{k=1\}}+|b|_1\int_{t_i}^{t}\int_0^1G_{t-s}(x,y)\|D^k\varphi_s^y(t_i,Y_i^\tau)\|_{L^p(\Omega,\uH^{\otimes k})}\ud y\ud s \\
&+\int_{t_i}^t\int_0^1G_{t-s}(x,y)\sum_{j=1}^{k-1}\left(\begin{array}{c}k-1 \\j\end{array}\right)
\|D^jb^\prime(\varphi_s^y(t_i,Y_i^\tau))\|_{L^{2p}(\Omega,\uH^{\otimes j})}
\|D^{k-j}\varphi_s^y(t_i,Y_i^\tau)\|_{L^{2p}(\Omega,\uH^{\otimes {k-j}})}\ud y\ud s .
\end{align*}
The Fa\`a di Bruno's formula (see e.g. \cite{SK05}) gives  that 
\begin{align*}
D^jb^\prime(\varphi_s^y(t_i,Y_i^\tau))=\sum \frac{j !}{l_{1} ! \cdots l_{j} !}b^{(l+1)}(\varphi_s^y(t_i,Y_i^\tau))\left(\frac{D \varphi_s^y(t_i,Y_i^\tau)}{1 !}\right)^{l_{1}} \cdots\left(\frac{D^{j} \varphi_s^y(t_i,Y_i^\tau)}{j !}\right)^{l_{j}},
\end{align*}
where $l=l_1+\cdots+l_j$ and the sum is taken over all partitions of $j$, i.e., values of $l_1,\cdots,l_j$ such that $l_1+2l_2+\cdots+jl_j=j$.  Using H\"older's inequality, for $1/p_1+\cdots+1/p_{j}=1/p$, we have
\begin{align*}
&\|D^jb^\prime(\varphi_s^y(t_i,Y_i^\tau))\|_{L^{p}(\Omega,{\uH^{\otimes j}})}\\
\le &C\sum \frac{j!}{l_{1} ! \cdots l_{j} !}\left\|\frac{D \varphi_s^y(t_i,Y_i^\tau)}{1 !}\right\|^{l_1}_{L^{l_1p_1}(\Omega,{\uH})} \cdots\left\|\frac{D^{j} \varphi_s^y(t_i,Y_i^\tau)}{j !}\right\|^{l_j}_{L^{l_jp_j}(\Omega,{\uH^{\otimes j}})}.
\end{align*}
Therefore, by the assumption that \eqref{Ykp} holds up to the index $k-1$, we arrive at
\begin{align}\nonumber
&\|D^k\varphi_t^x(t_i,Y_i^\tau)\|_{L^p(\Omega,\uH^{\otimes k})}\le C+\int_0^1G_{t-t_i}(x,y)\|D^kY_i^\tau(y)\|_{L^p(\Omega,\uH^{\otimes k})}\ud y\\\label{DkYi}
&+C\int_{t_i}^{t}\int_0^1G_{t-s}(x,y)\|D^k\varphi_s^y(t_i,Y_i^\tau)\|_{L^p(\Omega,\uH^{\otimes k})}\ud y\ud s.
\end{align}
By Leibnitz's rule, 
\begin{align*}
&D^k\Phi^y_{t_i}(0,u_0))=\sum_{j=0}^{i-1}\int_{t_j}^{t_{j+1}}\int_{0}^{1}G_{t_i-r}(y,z)b^{\prime}(\Phi^z_{t_j}(0,u_0))D^{k} \Phi^z_{t_j}(0,u_0)\ud z\ud r\\
&+\sum_{j=0}^{i-1}\int_{t_j}^{t_{j+1}}\int_{0}^{1}G_{t_i-r}(y,z)\sum_{m=1}^{k-1}\left(\begin{array}{c}k-1 \\m\end{array}\right)
D^mb^\prime(\Phi^z_{t_j}(0,u_0))
D^{k-m}\Phi^z_{t_j}(0,u_0)
\ud z\ud r\\
&+D\int_{0}^{t_i}\int_{0}^{1}G_{t_i-r}(y,z)\sigma W(\ud r,\ud z)\cdot\mathbf{1}_{\{k=1\}}.
\end{align*}
Similar to the proof of \eqref{DkYi}, the Fa\`a di Bruno's formula and the assumption that \eqref{Yiy} holds up to $k-1$ imply
\begin{align}\label{DKN}
&\|D^k\Phi^y_{t_i}(0,u_0))\|_{L^p(\Omega,\uH^{\otimes k})}\le C+|b|_1 \sum_{j=0}^{i-1}\int_{t_j}^{t_{j+1}}\int_{0}^{1}G_{t_i-r}(y,z)\|D^{k} \Phi^z_{t_j}(0,u_0)\|_{L^p(\Omega,\uH^{\otimes k})}\ud z\ud r
\end{align}
and 
\begin{align}\nonumber\label{DKE}
&\|D^k\varphi_{t_i}^y(t_{i-1},\Phi_{t_{i-1}}(0,u_0))\|_{L^p(\Omega,\uH^{\otimes k})}\le C+\int_0^1G_{\delta }(y,z)\|\Phi^z_{t_{i-1}}(0,u_0)\|_{L^p(\Omega,\uH^{\otimes k})}\ud z\\
&+C\int_{t_{i-1}}^{t_i}\int_0^1G_{t_i-r}(y,z)\|D^k\varphi_r^z(t_{i-1},\Phi_{t_{i-1}}(0,u_0))\|_{L^p(\Omega,\uH^{\otimes k})}\ud z\ud r.
\end{align}
Taking supremum over $y\in(0,1)$ on both sides of \eqref{DKN} and \eqref{DKE}, then applying the Gronwall lemma, we arrive at
\begin{align*}
&\sup_{y\in(0,1)}\|D^k\Phi^y_{t_i}(0,u_0))\|_{L^p(\Omega,\uH^{\otimes k})}\le C,\,\forall\,i=1,\cdots,N
\end{align*}
and
\begin{align*}
&\sup_{y\in(0,1)}\|D^k\varphi_{t_i}^y(t_{i-1},\Phi_{t_{i-1}}(0,u_0))\|_{L^p(\Omega,\uH^{\otimes k})}\le C,\,\forall\,i=1,\cdots,N,
\end{align*}
which, together with \eqref{Fkp}, completes the proof of \eqref{Yiy}. Finally, it follows from \eqref{Yiy} and \eqref{DkYi} that \eqref{Ykp} holds for $k$ and the proof is completed.

\end{proof}
The proof of Lemma \ref{Ykp0} is naturally extended to the following cases, whose proof is skipped.
\begin{cor}\label{ZRY}
Assume that $b\in\mathcal C_{\mathbf b}^\infty$. Then for any integers $k\ge0,\,p\ge1$, there exists $C=C(k,p,T,b,\sigma,\|u_0\|_E)$ such that for any $\tau,\,\beta\in[0,1],$ we have
 \begin{align*}
& \sup_{i=1,\cdots,N}\sup_{\theta_1\in(t_i,T],\,z\in(0,1)}\|b^{\prime}(\varphi_{\theta_1}^z(t_i,Y_i^\tau))\|_{k,p}\le C,\\
 &\sup_{i=1,\cdots,N}\sup_{r\in(t_{i-1},t_i],\,y\in(0,1)}\|b^{\prime}(Z_i^\beta(r,y))\|_{k,p}\le C,\\
  &\sup_{i=1,\cdots,N}\sup_{\theta\in(0,t_{i-1}],\,\xi\in(0,1)}\left\|b\left(\Phi_{\lfloor\theta\rfloor}^\xi(0,u_0)\right)\right\|_{k,p}\le C.
 \end{align*}
\end{cor}

Based on Proposition \ref{inverse} and Lemma \ref{Ykp0}, we are in a position to show the regularity of the density of $u^\delta(T,x)$ and to give the proof of Lemma \ref{G11}.
\begin{tho}\label{smooth}
Assume that $b\in\mathcal C_{\mathbf b}^\infty$ and $\delta \in\left(0,\frac{T}{12}\wedge\frac{\log\frac{3}{2}}{4|b|_1}\right)$. Then
 for every $x\in(0,1)$,  $u^\delta(T,x)$ admits an infinitely differentiable density.
 \end{tho}
\begin{proof} 
In view of Proposition \ref{inverse} and \eqref{Ykp},
 for every $x\in(0,1),\,i\in\{1,\cdots,N\}$ and $\tau\in[0,1]$, $\varphi_T^x(t_i,Y_i^\tau)$ is non-degenerate and so is $u^\delta(T,x)$. Consequently,  a direct application of the Bouleau--Hirsch's criterion (see e.g. \cite[Theorem 2.1.4]{DN06}) yields that  for every $x\in(0,1)$, $u^\delta(T,x)$ admits an infinitely differentiable density.
\end{proof}
 We emphasize that Lemma \ref{G11} will be used repeatedly to ensure that the generic constant $C$ appeared in Theorem \ref{err} is independent of the test function $f$.

\textit{Proof of Lemma \ref{G11}} :
Invoking Proposition \ref{inverse} and Lemma \ref{Ykp0}, it follows from \cite[Proposition 2.1.4]{DN06} that for any $\alpha\in\mathbb{N}$, $k\ge1$, there exists an element $H_{\alpha+1}(\varphi_T^x(t_i,Y_i^\tau),G_1)\in\D^\infty$ such that
\begin{align}\label{IBP}
\E\left[f^{(\alpha)}(\varphi_T^x(t_i,Y_i^\tau))G_1\right]=\E\left[F(\varphi_T^x(t_i,Y_i^\tau))H_{\alpha+1}(\varphi_T^x(t_i,Y_i^\tau),G_1)\right].
\end{align}
Furthermore, for $p_1\ge 1$, there exist constants $C(p_1,\alpha),\,a,\,q,\,k^\prime,\,p^\prime,\,w,\,k,\,p$ such that 
\begin{align*}
\|H_{\alpha+1}(\varphi_T^x(t_i,Y_i^\tau),G_1)\|_{p_1}\le C(p_1,\alpha)\|\Gamma_{\varphi_T^x(t_i,Y_i^\tau)}^{-1}\|_q^a\|\varphi_T^x(t_i,Y_i^\tau)\|_{k^\prime,p^\prime}^{w}\|G_1\|_{k,p}.
\end{align*}
Hence, by $0\le F\le 1$, Proposition \ref{inverse} and Lemma \ref{Ykp0}, we complete the proof.
\qed

%

\subsection{Regularity of derivatives}
In this part, we present Lemma \ref{MD0} on the moments of the Gateaux derivative and Lemma \ref{Ykp11} on the moments of the Malliavin derivative of $\varphi_t^x(t_i,Y_i^\tau)$, which will be used in the proof of Theorem \ref{err}. As we will see, 
the $p$-th moment of these derivatives are dominated by the corresponding Green function, instead of being bounded by a constant. This is the main difference in the weak convergence analysis  between stochastic partial differential equations and stochastic ordinary differential equations.

 \begin{lem}\label{MD0}
 Assume that $b\in\mathcal C_{\mathbf b}^\infty$. Then
for any  integers $k\ge0,\,p\ge 1$, there exists  $C=C(k,p,T,b,\sigma)$ such that
\begin{align}\label{G1}
&\|\la\mathcal{D}\varphi_t^x(t_i,Y_i^\tau),G_{t_i-r}(\cdot,y)\ra\|_{k,p}\le CG_{t-r}\left(x,y\right)
\end{align}
holds for every $r\in[t_{i-1},t_i),\,t_i\le t\le T,\,i\in\{1,\cdots,N\}$ and $\tau,x,\,y\in(0,1)$.
\end{lem}
\begin{proof}
The proof is completed by induction on $k$. 
From \eqref{Dvar}, 
the Minkowskii's inequality and the boundedness of $b^\prime$ give that
\begin{align*}
\|\la\mathcal{D}\varphi_t^x(t_i,Y_i^\tau),G_{t_i-r}(\cdot,y)\ra\|_p&\le G_{t-r}\left(x,y\right)\\
&+|b|_1\int_{t_i}^{t}\int_{0}^{1}G_{t-\theta_1}(x,z)\|\la\mathcal{D}\varphi_{\theta_1}^z(t_i,Y_i^\tau),G_{t_i-r}(\cdot,y)\ra \|_p\ud z\ud \theta_1.
\end{align*}
A direct application of Lemma \ref{GW} completes the proof of \eqref{G1} when $k=0$.

Assuming that \eqref{G1} holds up to the index $k-1$, $k\ge1$. Hence, by applying Leibnitz's rule, H\"older's inequality and Corollary \ref{ZRY}, it holds for $t_i<\theta_1<t$ that
\begin{align*}
& \|D^k\{b^{\prime}(\varphi_{\theta_1}^z(t_i,Y_i^\tau))\la\mathcal{D}\varphi_{\theta_1}^z(t_i,Y_i^\tau),G_{t_i-r}(\cdot,y)\ra\}\|_{L^p(\Omega,\uH^{\otimes k})}\\
 \le &|b|_1 \|D^k\la\mathcal{D}\varphi_{\theta_1}^z(t_i,Y_i^\tau),G_{t_i-r}(\cdot,y)\ra\|_{L^p(\Omega,\uH^{\otimes k})}\\
 &+C(k,p,T)\|b^{\prime}(\varphi_{\theta_1}^z(t_i,Y_i^\tau))\|_{k,2p}\|\la\mathcal{D}\varphi_{\theta_1}^z(t_i,Y_i^\tau),G_{t_i-r}(\cdot,y)\ra\|_{k-1,2p}\\
  \le& |b|_1 \|D^k\la\mathcal{D}\varphi_{\theta_1}^z(t_i,Y_i^\tau),G_{t_i-r}(\cdot,y)\ra\|_{L^p(\Omega,\uH^{\otimes k})}+C(k,p,T)G_{\theta_1-r}\left(z,y\right),
\end{align*}
which, together with the semigroup property of $G$, indicates that
\begin{align*}
&\|D^k\la\mathcal{D}\varphi_t^x(t_i,Y_i^\tau),G_{t_i-r}(\cdot,y)\ra\|_{L^p(\Omega,\uH^{\otimes k})}\le C(T,k,p)G_{t-r}\left(x,y\right)\\
&+|b|_1 \int_{t_i}^{t}\int_{0}^{1}G_{t-\theta_1}(x,z)\| D^k\la\mathcal{D}\varphi_{\theta_1}^z(t_i,Y_i^\tau),G_{t_i-r}(\cdot,y)\ra \|_{L^p(\Omega,\uH^{\otimes k})}\ud z\ud \theta_1.
\end{align*}
Consequently, \eqref{G1} is valid for $k$ thanks to Lemma \ref{GW} and \eqref{Fkp} and the proof is completed.
\end{proof}

\begin{lem}\label{Ykp11}
Assume that $b\in\mathcal C_{\mathbf b}^\infty$. Then
for any integers $k\ge0,\,p\ge1$, there exists $C=C(k,p,T,b,\sigma)$ such that for every $i\in\{1,\cdots,N\}$, $\theta\in(0,t_{i-1}),\tau,x,y,\xi\in(0,1)$ and $t\in(t_i,T],$
\begin{align}\label{Yiy1}
&\|D_{\theta,\xi}\Phi^y_{t_i}(0,u_0)\|_{k,p}+\|D_{\theta,\xi}\varphi_{t_i}^y(t_{i-1},\Phi_{t_{i-1}}(0,u_0))\|_{k,p}\le CG_{t_i-\theta}(y,\xi),\\\label{Ykp1}
&\|D_{\theta,\xi}\varphi_t^x(t_i,Y_i^\tau)\|_{k,p}\le CG_{t-\theta}(x,\xi).
\end{align}
\end{lem}

\begin{proof}
We proceed by induction on $k$, which is 
analogous to the proof of Lemma \ref{Ykp0} and Lemma \ref{MD0}. Thus, we only give the details of the proof of the case $k=0$, and the induction argument for $k\ge1$ is omitted.

 Let $y\in(0,1)$ and $i\in\{1,\cdots,N\}$ be arbitrarily fixed.
First, we claim that
\begin{align}\label{Phik=0}
\|D_{\theta,\xi}\Phi^y_{t_i}(0,u_0)\|_p\le CG_{t_i-\theta}(y,\xi),\,\forall\,\theta\in(0,t_{i-1}),
\,\xi\in(0,1).
\end{align}
In fact, if $i=2$, then
for any $\theta\in(0,t_1),\,\xi\in(0,1),\,D_{\theta,\xi}\Phi^y_{t_1}(0,u_0)=\sigma G_{t_1-\theta}(y,\xi)$ and \begin{align*}
D_{\theta,\xi}\Phi^y_{t_2}(0,u_0)=\sigma G_{t_2-\theta}(y,\xi)+\int_{t_1}^{t_2}\int_{0}^{1}G_{t_2-r}(y,z)b^\prime(\Phi^z_{t_1}(0,u_0))D_{\theta,\xi}\Phi^z_{t_1}(0,u_0)\ud z\ud r.
\end{align*}
Taking the norm $\|\cdot\|_p$ on both sides and by \eqref{Gts},
\begin{align*}
&\int_{t_1}^{t_2}\int_{0}^{1}G_{t_2-r}(y,z)G_{t_1-\theta}(z,\xi)\ud z\ud r\\=&\int_{t_1}^{t_2}G_{t_2-r+t_1-\theta}(y,\xi)\ud r\le C(T)\int_{t_1}^{t_2}\sqrt{\frac{t_2-r+t_1-\theta}{t_2-\theta}}G_{t_2-\theta}(y,\xi)\ud r\le C(T)\delta G_{t_2-\theta}(y,\xi)
\end{align*}
because of $\theta\in(0,t_1)$ and $t_2=2\delta $,
then \eqref{Phik=0} holds true when $i=2$.
To show \eqref{Phik=0} for general $3\le i\le N$, by induction, we assume that \eqref{Phik=0}
holds up to $i-1$.  Now assume that $\theta\in(0,t_{i-1})$ and $\xi\in(0,1)$. Then there exists $k\in\{1,\cdots,i-1\}$ such that $\theta\in(t_{k-1},t_{k}]$ and 
\begin{align*}
\|D_{\theta,\xi}\Phi^y_{t_i}(0,u_0)\|_p\le&|b|_1\sum_{j=k}^{i-1}\int_{t_j}^{t_{j+1}}\int_{0}^{1}G_{t_i-r}(y,z)\|D_{\theta,\xi}\Phi^z_{t_j}(0,u_0)\|_p\ud z\ud r+G_{t_i-\theta}(y,\xi)|\sigma|,\\
\le&C\sum_{j=k}^{i-1}\int_{t_j}^{t_{j+1}}\int_{0}^{1}G_{t_i-r}(y,z)G_{t_j-\theta}(z,\xi)\ud z\ud r+G_{t_i-\theta}(y,\xi)|\sigma|\\
\le&C\sum_{j=k}^{i-1}\int_{t_j}^{t_{j+1}}G_{t_i-r+t_j-\theta}(y,\xi)\ud r+G_{t_i-\theta}(y,\xi)|\sigma|,
\end{align*}
where
\begin{align*}
\sum_{j=k}^{i-1}\int_{t_j}^{t_{j+1}}G_{t_i-r+t_j-\theta}(y,\xi)\ud r\le& C(T)\sum_{j=k}^{i-1}\int_{t_j}^{t_{j+1}}\sqrt{\frac{t_i-\theta}{t_i-r+t_j-\theta}}\ud rG_{t_i-\theta}(y,\xi)\\
=&2C(T)\delta \sum_{j=k-1}^{i-1}\frac{\sqrt{t_i-\theta}}{\sqrt{t_i-\theta}+\sqrt{t_{i-1}-\theta}}G_{t_i-\theta}(y,\xi)\le C(T)G_{t_i-\theta}(y,\xi),
\end{align*}
in view of \eqref{Gts}.
This completes the proof of $\eqref{Phik=0}$. 
Notice that for $\theta\in(t_{i-2},t_{i-1})$, $$\|D_{\theta,\xi}\Phi^y_{t_{i-1}}(0,u_0)\|_p= |\sigma| G_{t_{i-1}-\theta}(y,\xi)\le CG_{t_i-\theta}(y,\xi),$$
and by \eqref{Phik=0}, for $\theta\in(0,t_{i-2})$, $$\|D_{\theta,\xi}\Phi^y_{t_{i-1}}(0,u_0)\|_p\le CG_{t_i-\theta}(y,\xi).$$
 Hence, by the semigroup property of $G$, we have
\begin{align*}
&\|D_{\theta,\xi}\varphi_{t_i}^y(t_{i-1},\Phi_{t_{i-1}}(0,u_0))\|_p\\
\le &CG_{t_i-\theta}(y,\xi)+C\int_{t_{i-1}}^{t_i}\int_0^1G_{t_i-r}(y,z)\|D_{\theta,\xi}\varphi_r^z(t_{i-1},\Phi_{t_{i-1}}(0,u_0))\|_p\ud z\ud r.
\end{align*}
By Lemma \ref{GW} and \eqref{Phik=0}, we complete the proof of \eqref{Yiy1} when $k=0$.
Finally, the definition of $Y_i^\tau$ implies that for any $\theta\in(0,t_{i-1}),\,\xi\in(0,1)$, $\|D_{\theta,\xi}Y_i^\tau(y)\|\le CG_{t_i-\theta}(y,\xi)$ and thereby \eqref{Ykp1} follows from an analogue argument by using Lemma \ref{GW}.
\end{proof}

\begin{cor}\label{DZ}
Assume that $b\in\mathcal C_{\mathbf b}^\infty$. Then for any integers $k\ge0$ and $p\ge1$, there exists some constant $C=C(k,p,T,b,\sigma)$ such that for every $i\in\{1,\cdots,N\}$, $t_{i-1}< r\le t_i$, $\theta\in(0,r)$ and $\beta\in[0,1]$, it holds that
\begin{align*}
\|D_{\theta,\xi}Z_i^\beta(r,y)\|_{k,p}\le CG_{r-\theta}(y,\xi).
 \end{align*}
\end{cor}

\section{Weak convergence analysis}\label{S4}

\subsection{Test function-independent analysis}
In this part, we give the proof of Theorem \ref{err}, which is essential to obtain the convergence order of density approximations in the uniform convergence topology.


\textit{Proof of Theorem \ref{err}}:
 Observing that
 \begin{equation*}
\E[f(u^\delta(T,x))]- \E\left[f(u(T,x))\right]=\E\left[f(\varphi_T^x(0,u_0))\right]-\E\left[f(\Phi_T^x(0,u_0))\right],
 \end{equation*}
 we  proceed to estimate the summation $\sum_{i=1}^N\E[\mathcal I_i^j],\,j\in\{1,\cdots,5\}$, defined in \eqref{weak error}.

\subsubsection{Estimate of $\mathcal I^1_i$}
For fixed $0\le r< t_i\le T$ and $y\in (0,1)$, we have 
$$G_{t_i-r}(\cdot,y)\in E.$$
Invoking \cite[Proposition 1.5.6]{DN06}, Lemma \ref{MD0} and Corollary \ref{ZRY},  for any $k,\,p$,
\begin{align*}
&\|\la\mathcal{D}\varphi_T^x(t_i,Y_i^\tau),G_{t_i-r}(\cdot,y)\ra b^\prime(Z_i^\beta(r,y))\|_{k,p}\\\le &C(k,p)\|\la\mathcal{D}\varphi_T^x(t_i,Y_i^\tau),G_{t_i-r}(\cdot,y)\ra\|_{k,2p}\|b^\prime(Z_i^\beta(r,y))\|_{k,2p}\le C(k,p,T)G_{T-r}(x,y),
\end{align*}
which, combined with Lemma \ref{G11} and Lemma \ref{Green} implies that
\begin{align*}
\sum_{i=2}^N\left|\E\left[\mathcal I^1_i\right]\right|\le&\sum_{i=1}^N\int_0^1\int_0^1\int^{t_i}_{t_{i-1}}\int_{0}^{1}\left|\E\left[f^\prime(\varphi_T^x(t_i,Y_i^\tau))\la\mathcal{D}\varphi_T^x(t_i,Y_i^\tau),G_{t_i-r}(\cdot,y)\ra
 b^\prime(Z_i^\beta(r,y))\right]\right|\\&\int_0^1|G_r(y,\xi)-G_{t_{i-1}}(y,\xi)||u_0(\xi)|\ud\xi\ud y\ud r\ud\beta\ud\tau\\
 \le&C(k,p,T)\|u_0\|_E\sum_{i=2}^N\int^{t_i}_{t_{i-1}}(r-t_{i-1})^\nu(t_{i-1})^{-\nu}\ud r
 \le C\delta^\nu\int_{\delta }^{T}\frac{1}{r^\nu}\ud r\le C\delta^\nu
\end{align*}
with  $\nu\in(\frac{1}{3},1)$. In addition, for $i=1$,
\begin{align*}
\left|\E\left[\mathcal I^1_1\right]\right|\le&C\int^{\delta }_{0}\int_{0}^{1}G_{T-r}(x,y)\left|\int_0^1G_r(y,\xi)u_0(\xi)\ud\xi-u_0(y)\right|\ud y\ud r\le 2C\|u_0\|_E\delta .
\end{align*}

\subsubsection{Estimate of $\mathcal {I}^2_i$}
Similarly, we again apply Lemma \ref{Green} and Lemma \ref{G11}, Lemma \ref{MD0} and Corollary \ref{ZRY} to obtain
\begin{align*}
\sum_{i=1}^N\left|\E\left[\mathcal {I}^2_i\right]\right|\le&\sum_{i=1}^N\int_0^1\int_0^1\int^{t_i}_{t_{i-1}}\int_{0}^{1}  \int_0^{t_{i-1}}\int_0^1\Big|\E\Big[f^\prime(\varphi_T^x(t_i,Y_i^\tau))\la\mathcal{D}\varphi_T^x(t_i,Y_i^\tau),G_{t_i-r}(\cdot,y)\ra\\&
b^\prime(Z_i^\beta(r,y))b\left(\Phi_{\lfloor\theta\rfloor}^\xi(0,u_0)\right)\Big]\Big||G_{r-\theta}(y,\xi)-G_{t_{i-1}-\theta}(y,\xi)|\ud\xi\ud\theta\ud y\ud r\ud\beta\ud\tau\\
\le&C\sum_{i=1}^N\int^{t_i}_{t_{i-1}}\int_{0}^{1}\int_0^{t_{i-1}}G_{T-r}(x,y)\int_0^1|G_{r-\theta}(y,\xi)-G_{t_{i-1}-\theta}(y,\xi)|\ud\xi\ud\theta\ud y\ud r\\
\le&C\sum_{i=1}^N\int^{t_i}_{t_{i-1}}\int_0^{t_{i-1}}(r-t_{i-1})^\nu(t_{i-1}-\theta)^{-\nu}\ud\theta\ud r\le C\delta^\nu
\end{align*}
with $\nu\in(\frac{1}{3},1)$.

\subsubsection{Estimate of $\mathcal {I}^3_i$}
In view of the Malliavin integration by parts formula and chain rule, $\E\left[\mathcal {I}^3_i\right]$ is further decomposed into
\begin{align*}
\E\left[\mathcal {I}^3_i\right]
=&\int_0^1\int_0^1\int^{t_i}_{t_{i-1}}\int_{0}^{1} \int_0^{t_{i-1}}\int_0^1\E\left[{f^{\prime\prime}(\varphi_T^x(t_i,Y_i^\tau))D_{\theta,\xi}\varphi_T^x(t_i,Y_i^\tau)\la\mathcal{D}\varphi_T^x(t_i,Y_i^\tau),G_{t_i-r}(\cdot,y)\ra}\right.\\&\left.{ b^\prime(Z_i^\beta(r,y))}\right]\{G_{r-\theta}(y,\xi)-G_{t_{i-1}-\theta}(y,\xi)\}\sigma\ud\xi\ud\theta \ud y\ud r\ud\beta\ud\tau\\
 &+\int_0^1\int_0^1\int^{t_i}_{t_{i-1}}\int_{0}^{1} \int_0^{t_{i-1}}\int_0^1\E\left[{f^\prime(\varphi_T^x(t_i,Y_i^\tau))D_{\theta,\xi}\la\mathcal{D}\varphi_T^x(t_i,Y_i^\tau),G_{t_i-r}(\cdot,y)\ra}\right.\\&\left.{
b^\prime(Z_i^\beta(r,y))}\right]
\{G_{r-\theta}(y,\xi)-G_{t_{i-1}-\theta}(y,\xi)\}\sigma\ud\xi\ud\theta \ud y\ud r\ud\beta\ud\tau\\
  &+\int_0^1\int_0^1\int^{t_i}_{t_{i-1}}\int_{0}^{1} \int_0^{t_{i-1}}\int_0^1\E\left[{f^\prime(\varphi_T^x(t_i,Y_i^\tau))\la\mathcal{D}\varphi_T^x(t_i,Y_i^\tau),G_{t_i-r}(\cdot,y)\ra
b^{\prime\prime}(Z_i^\beta(r,y))}\right.\\&\left.{D_{\theta,\xi}Z_i^\beta(r,y)}\right]\{G_{r-\theta}(y,\xi)-G_{t_{i-1}-\theta}(y,\xi)\}\sigma\ud\xi\ud\theta \ud y\ud r\ud\beta\ud\tau=:
\mathcal J^i_1+\mathcal J^i_2+\mathcal J^i_3.
\end{align*}

\textsf{Estimate of $\mathcal {J}^i_1$:}
By Lemma \ref{G11}, it holds that
\begin{align*}
\left|\mathcal J_1^i\right|\le\int_0^1\int_0^1\int^{t_i}_{t_{i-1}}\int_{0}^{1} \int_0^{t_{i-1}}\int_0^1&C\|D_{\theta,\xi}\varphi_T^x(t_i,Y_i^\tau)\la\mathcal{D}\varphi_T^x(t_i,Y_i^\tau),G_{t_i-r}(\cdot,y)\ra b^\prime(Z_i^\beta(r,y))\|_{k,p}\\&|G_{r-\theta}(y,\xi)-G_{t_{i-1}-\theta}(y,\xi)|\ud\xi\ud\theta \ud y\ud r\ud\beta\ud\tau,
\end{align*}
for some positive constants $k,\,p$.
Hence, applying \cite[Proposition 1.5.6]{DN06}, the semigroup property of $G$, \eqref{Gts0} and Lemma \ref{Green}, we have
\begin{align*}
\sum_{i=1}^{N}\left|\mathcal J_1^i\right|\le&C \sum_{i=1}^{N}\int^{t_i}_{t_{i-1}}\int_{0}^{1} \int_0^{t_{i-1}}\int_0^1G_{T-\theta}(x,\xi)G_{T-r}(x,y)
|G_{r-\theta}(y,\xi)-G_{t_{i-1}-\theta}(y,\xi)|\ud\xi\ud\theta \ud y\ud r\\
\le&C \sum_{i=1}^{N}\int^{t_i}_{t_{i-1}}\int_{0}^{1} \int_0^{t_{i-1}}(T-r)^{-\frac{1}{2}}G_{T-\theta}(x,\xi)
\int_0^1|G_{r-\theta}(y,\xi)-G_{t_{i-1}-\theta}(y,\xi)|\ud y\ud\theta \ud \xi\ud r\\
\le& C\delta ^{\nu}\sum\nu_{i=1}^{N}\int^{t_i}_{t_{i-1}}\int_0^{t_{i-1}} (T-r)^{-\frac{1}{2}}
(t_{i-1}-\theta)^{-\nu}\ud\theta \ud r\\
\le& C\delta ^{\nu}\sum_{i=1}^{N}\int^{t_i}_{t_{i-1}} (T-r)^{-\frac{1}{2}} \ud r
\int_0^{t_{i-1}}(t_{i-1}-\theta)^{-\nu}\ud\theta\le C\delta ^{\nu}
\end{align*}
with $\nu\in(\frac{1}{3},1)$.

\textsf{Estimate of $\mathcal {J}_2^i$:}
To treat $\mathcal {J}_2^i$, notice that by Lemma \ref{G11}, there exist some constants $k,\,p$, such that
\begin{align*}
\left|\mathcal J^i_2\right|
\le &C\int_0^1\int^{t_i}_{t_{i-1}}\int_{0}^{1} \int_0^{t_{i-1}}\int_0^1\|D_{\theta,\xi}\la\mathcal{D}\varphi_T^x(t_i,Y_i^\tau),G_{t_i-r}(\cdot,y)\ra
\|_{k,p}\\&|G_{r-\theta}(y,\xi)-G_{t_{i-1}-\theta}(y,\xi)|\ud\xi\ud\theta \ud y\ud r\ud\tau
=:C\int_0^1A_i(T,x;k,p,\nu,\tau)\ud\tau.
\end{align*}
Now we proceed to show that for any $\nu\in(\frac{1}{3},1)$, there exists $C=C(T,k,p,\nu)$ such that for any $x\in(0,1)$, $t\in(t_i,T]$ and $\tau\in(0,1)$,
\begin{equation}\label{KTx}
A_i(t,x;k,p,\nu,\tau)\le C\delta^{1+\nu}.
\end{equation}
First, by chain rule and the semigroup property of $G$, we obtain
\begin{align}\nonumber
\la\mathcal{D}\varphi_t^x(t_i,Y_i^\tau)&,G_{t_i-r}(\cdot,y)\ra=G_{t-r}\left(x,y\right)\\\label{Dvar}
&+\int_{t_i}^{t}\int_{0}^{1}G_{t-\theta_1}(x,z)b^{\prime}(\varphi_{\theta_1}^z(t_i,Y_i^\tau))\la\mathcal{D}\varphi_{\theta_1}^z(t_i,Y_i^\tau),G_{t_i-r}(\cdot,y)\ra\ud z\ud \theta_1.
\end{align}
Taking the Malliavin derivative $D_{\theta,\xi}$ on both sides of \eqref{Dvar} gives
\begin{align*}\nonumber
&D_{\theta,\xi}\la\mathcal{D}\varphi_t^x(t_i,Y_i^\tau),G_{t_i-r}(\cdot,y)\ra\\
=&\int_{t_i}^{t}\int_{0}^{1}G_{t-\theta_1}(x,z)D_{\theta,\xi}b^{\prime}(\varphi_{\theta_1}^z(t_i,Y_i^\tau))\la\mathcal{D}\varphi_{\theta_1}^z(t_i,Y_i^\tau),G_{t_i-r}(\cdot,y)\ra\ud z\ud \theta_1\\&+\int_{t_i}^{t}\int_{0}^{1}G_{t-\theta_1}(x,z)b^{\prime}(\varphi_{\theta_1}^z(t_i,Y_i^\tau))D_{\theta,\xi}\la\mathcal{D}\varphi_{\theta_1}^z(t_i,Y_i^\tau),G_{t_i-r}(\cdot,y)\ra\ud z\ud \theta_1.
\end{align*}
Then it follows from the definition of $A_i(t,x;0,p,\nu,\tau)$ and H\"older's inequality that
\begin{align*}
A_i(t,x;0,p,\nu,\tau)&\le\int^{t_i}_{t_{i-1}}\int_{0}^{1} \int_0^{t_{i-1}}\int_0^1\int_{t_i}^{t}\int_{0}^{1}G_{t-\theta_1}(x,z)\|D_{\theta,\xi}b^{\prime}(\varphi_{\theta_1}^z(t_i,Y_i^\tau))\|_{p_1}\\&\|\la\mathcal{D}\varphi_{\theta_1}^z(t_i,Y_i^\tau),G_{t_i-r}(\cdot,y)\|_{p_2}|G_{r-\theta}(y,\xi)-G_{t_{i-1}-\theta}(y,\xi)|\ud z\ud \theta_1\ud\xi\ud\theta \ud y\ud r\\&+\int_{t_i}^{t}\int_{0}^{1}G_{t-\theta_1}(x,z)|b|_1A_i(\theta_1,z;0,p,\nu,\tau)\ra\ud z\ud \theta_1
\end{align*}
with $1/p_1+1/p_2=1/p$.
Lemmas  \ref{MD0}, \ref{Ykp11} and \ref{Green}  yield that
\begin{align*}
&A_i(t,x;0,p,\nu,\tau)\\
\le& C\int_{t_i}^{t}\int_{0}^{1}G_{t-\theta_1}(x,z)A_i(\theta_1,z;0,p,\nu,\tau)
\ud z\ud\theta_1+\int_{t_i}^{t}\int_{0}^{1}\int^{t_i}_{t_{i-1}}\int_{0}^{1} \int_0^{t_{i-1}}\int_0^1\\
& G_{t-\theta_1}(x,z)G_{\theta_1-\theta}(z,\xi)G_{\theta_1-r}(z,y)|G_{r-\theta}(y,\xi)-G_{t_{i-1}-\theta}(y,\xi)|\ud\xi\ud\theta \ud y\ud r\ud z\ud \theta_1\\
\le&\int_{t_i}^{t}\int_{0}^{1}G_{t-\theta_1}(x,z)A_i(\theta_1,z;0,p,\nu,\tau)
\ud z\ud\theta_1+\int_{t_i}^{t}\int_{0}^{1}\int^{t_i}_{t_{i-1}} \int_0^{t_{i-1}}\\&\frac{1}{\sqrt{t-\theta_1}}\frac{1}{\sqrt{\theta_1-\theta}}\int_{0}^{1}G_{\theta_1-r}(z,y) \ud y\int_0^1|G_{r-\theta}(y,\xi)-G_{t_{i-1}-\theta}(y,\xi)|\ud\xi\ud\theta\ud r\ud z\ud \theta_1\\
\le&\int_{t_i}^{t}\int_{0}^{1}G_{t-\theta_1}(x,z)A_i(\theta_1,z;0,p,\nu,\tau)
\ud z\ud\theta_1\\&+C\delta^\nu\int^{t_i}_{t_{i-1}} \int_{t_i}^{t}\frac{1}{\sqrt{t-\theta_1}}\frac{1}{\sqrt{\theta_1-t_i}}\ud\theta_1\int_0^{t_{i-1}}(t_{i-1}-\theta)^{-\nu}\ud\theta \ud r
\\
\le&\int_{t_i}^{t}\int_{0}^{1}G_{t-\theta_1}(x,z)A_i(\theta_1,z;0,p,\nu,\tau)
\ud z\ud\theta_1+C\delta^{1+\nu}.
\end{align*}
Taking supremum on both sides of the above inequality over $x\in(0,1),$ we complete the proof of \eqref{KTx} when $k=0$ by applying the Gronwall lemma. Assume that \eqref{KTx} holds up to the index $k-1$. Then the induction argument for general $k$ is similar to the proof of Lemma \ref{Ykp0} and thereby is omitted.

\textsf{Estimate of $\mathcal {J}_3^i$:}
By Lemmas \ref{G11}, \ref{MD0} and Corollaries \ref{ZRY}, \ref{DZ}, 
it holds for some constants $k,\,p$ that
\begin{align}\nonumber
\sum_{i=1}^{N}\left|\mathcal J^i_3\right|\le&\sum_{i=1}^N\int_0^1\int_0^1\int^{t_i}_{t_{i-1}}\int_{0}^{1} \int_0^{t_{i-1}}\int_0^1\|\la\mathcal{D}\varphi_T^x(t_i,Y_i^\tau),G_{t_i-r}(\cdot,y)\ra
b^{\prime\prime}(Z_i^\beta(r,y))D_{\theta,\xi}Z_i^\beta(r,y)\|_{k,p}\\\nonumber
&|G_{r-\theta}(y,\xi)-G_{t_{i-1}-\theta}(y,\xi)||\sigma|\ud\xi\ud\theta \ud y\ud r\ud\beta\ud\tau\\\nonumber
\le&C\sum_{i=1}^N\int^{t_i}_{t_{i-1}}\int_{0}^{1} \int_0^{t_{i-2}}\int_0^1G_{T-r}(x,y)G_{r-\theta}(y,\xi)
|G_{r-\theta}(y,\xi)-G_{t_{i-1}-\theta}(y,\xi)|\ud\xi\ud\theta \ud y\ud r\\\nonumber
+&C\sum_{i=1}^N\int^{t_i}_{t_{i-1}}\int_{0}^{1} \int_{t_{i-2}}^{t_{i-1}}\int_0^1G_{T-r}(x,y)G_{r-\theta}(y,\xi)
|G_{r-\theta}(y,\xi)-G_{t_{i-1}-\theta}(y,\xi)|\ud\xi\ud\theta \ud y\ud r\\\label{J3iN}
=&:\mathcal J_{31}+\mathcal J_{32}.
\end{align}
Then Lemma \ref{Green} with $\nu\in(\frac{1}{2},1)$ leads to
\begin{align*}
\mathcal J_{31}\le&C\sum_{i=1}^N\int^{t_i}_{t_{i-1}}\int_{0}^{1} \int_0^{t_{i-2}}G_{T-r}(x,y)(r-\theta)^{-\frac{1}{2}}
\int_0^1|G_{r-\theta}(y,\xi)-G_{t_{i-1}-\theta}(y,\xi)|\ud\xi\ud\theta \ud y\ud r\\
\le&C\delta ^{\nu}\sum_{i=1}^N\int^{t_i}_{t_{i-1}} \int_0^{t_{i-2}}(r-\theta)^{-\frac{1}{2}}(t_{i-1}-\theta)^{-\nu}\ud\theta \ud r\\
=&2C\delta^{\nu+1}\sum_{i=1}^N \int_0^{t_{i-2}}\frac{(t_{i-1}-\theta)^{-\nu}}{\sqrt{t_i-\theta}+\sqrt{t_{i-1}-\theta}}\ud\theta\\
\le&C\delta^{\nu+1}\sum_{i=1}^N \int_0^{t_{i-2}}(t_{i-1}-\theta)^{-\nu-\frac{1}{2}}\ud\theta\le C\delta^{\frac{1}{2}}.
\end{align*}
Besides, the positivity of $G$ leads to
\begin{align*}
\mathcal J_{32}\le&C\sum_{i=1}^N\int^{t_i}_{t_{i-1}}\int_{0}^{1} \int_{t_{i-2}}^{t_{i-1}}G_{T-r}(x,y)(r-\theta)^{-\frac{1}{2}}
\int_0^1|G_{r-\theta}(y,\xi)-G_{t_{i-1}-\theta}(y,\xi)|\ud\xi\ud\theta \ud y\ud r\\
\le&C\sum_{i=1}^N\int^{t_i}_{t_{i-1}}\int_{0}^{1} \int_{t_{i-2}}^{t_{i-1}}G_{T-r}(x,y)(r-\theta)^{-\frac{1}{2}}
\int_0^1G_{r-\theta}(y,\xi)+G_{t_{i-1}-\theta}(y,\xi)\ud\xi\ud\theta \ud y\ud r\\
\le&C\sum_{i=1}^N\int^{t_i}_{t_{i-1}}\int_{0}^{1} \int_{t_{i-2}}^{t_{i-1}}G_{T-r}(x,y)(r-\theta)^{-\frac{1}{2}}
\ud\theta \ud y\ud r\le C\delta^{\frac{1}{2}}.
\end{align*}

 \subsubsection{Estimate of $\mathcal {I}^4_i$}
 Similarly, we apply Lemmas \ref{G11}, \ref{MD0} and Corollary \ref{ZRY} to obtain that for some positive constants $k,\,p$,
 \begin{align*}
&\left|\E\left[\mathcal I^4_i\right]\right|\\
 \le&C\int_0^1\int_0^1\int^{t_i}_{t_{i-1}}\int_{0}^{1}\int_{t_{i-1}}^r\int_0^1\left\|\la\mathcal{D}\varphi_T^x(t_i,Y_i^\tau),G_{t_i-r}(\cdot,y)\ra b^\prime(Z_i^\beta(r,y))b\left(\varphi_\theta^\xi(t_{i-1},\Phi_{t_{i-1}}(0,u_0)\right)\right\|_{k,p}\\
 &\qquad  G_{r-\theta}(y,\xi)\ud \xi\ud\theta\ud y\ud r\ud\beta\ud\tau\\
 \le&C\int^{t_i}_{t_{i-1}}\int_{0}^{1}\int_{t_{i-1}}^r\int_0^1G_{T-r}(x,y)G_{r-\theta}(y,\xi)\ud \xi\ud\theta\ud y\ud r\le C\delta^2.
 \end{align*}
 \subsubsection{Estimate of $\mathcal {I}^5_i$}
We again apply the Malliavin integration by parts formula to decompose the term $\E[\mathcal {I}^5_i]$ into three parts
 \begin{align*} 
\E\left[\mathcal {I}^5_i\right]
=&\int_0^1\int_0^1\int^{t_i}_{t_{i-1}}\int_{0}^{1} \E\Big[f^\prime(\varphi_T^x(t_i,Y_i^\tau))\la\mathcal{D}\varphi_T^x(t_i,Y_i^\tau),G_{t_i-r}(\cdot,y)\ra b^\prime(Z_i^\beta(r,y))\\
&\int_{t_{i-1}}^r\int_0^1G_{r-\theta}(y,\xi)\sigma W(\ud \theta,\ud\xi)\Big]\ud y\ud r\ud\beta\ud\tau\\
=&\int_0^1\int_0^1\int^{t_i}_{t_{i-1}}\int_{0}^{1} \int_{t_{i-1}}^r\int_0^1\E\Big[f^{\prime\prime}(\varphi_T^x(t_i,Y_i^\tau))D_{\theta,\xi}\varphi_T^x(t_i,Y_i^\tau)\la\mathcal{D}\varphi_T^x(t_i,Y_i^\tau),G_{t_i-r}(\cdot,y)\ra \\
&b^\prime(Z_i^\beta(r,y))\Big]G_{r-\theta}(y,\xi)\ud\xi\ud \theta\ud y\ud r\ud\beta\ud\tau\\
&+\int_0^1\int_0^1\int^{t_i}_{t_{i-1}}\int_{0}^{1} \int_{t_{i-1}}^r\int_0^1\E\left[f^{\prime}(\varphi_T^x(t_i,Y_i^\tau))D_{\theta,\xi}\la\mathcal{D}\varphi_T^x(t_i,Y_i^\tau),G_{t_i-r}(\cdot,y)\ra b^\prime(Z_i^\beta(r,y))\right]\\
&G_{r-\theta}(y,\xi)\ud\xi\ud \theta\ud y\ud r\ud\beta\ud\tau\\
&+\int_0^1\int_0^1\int^{t_i}_{t_{i-1}}\int_{0}^{1} \int_{t_{i-1}}^r\int_0^1\E\left[f^{\prime}(\varphi_T^x(t_i,Y_i^\tau))\la\mathcal{D}\varphi_T^x(t_i,Y_i^\tau),G_{t_i-r}(\cdot,y)\ra D_{\theta,\xi}b^\prime(Z_i^\beta(r,y))\right]\\
&G_{r-\theta}(y,\xi)\ud\xi\ud \theta\ud y\ud r\ud\beta\ud\tau=:\sum_{j=1}^3\mathcal K^i_j.
\end{align*}

\textsf{Estimate of $\mathcal {K}_1^i$:}
By applying Lemma \ref{G11}, Lemma \ref{MD0}, Corollary \ref{ZRY} and the semigroup property of $G$, we obtain that for some positive constants $k,\,p$,
\begin{align*}
\sum_{i=1}^N\left|\mathcal K^i_1\right|
\le&C\sum_{i=1}^N\int_0^1\int_0^1\int^{t_i}_{t_{i-1}}\int_{0}^{1} \int_{t_{i-1}}^r\int_0^1\|D_{\theta,\xi}\varphi_T^x(t_i,Y_i^\tau)\la\mathcal{D}\varphi_T^x(t_i,Y_i^\tau),G_{t_i-r}(\cdot,y)\ra\\
& b^\prime(Z_i^\beta(r,y))\|_{k,p}G_{r-\theta}(y,\xi)\ud\xi\ud \theta\ud y\ud r\ud\beta\ud\tau\\
\le&C\sum_{i=1}^N\int^{t_i}_{t_{i-1}}\int_{0}^{1} \int_{t_{i-1}}^r\int_0^1G_{T-\theta}(x,\xi)G_{T-r}(x,y)
G_{r-\theta}(y,\xi)\ud\xi\ud \theta\ud y\ud r\\
\le&C\delta \sum_{i=1}^N\int^{t_i}_{t_{i-1}}\int_0^1G^2_{T-\theta}(x,\xi)\ud\xi\ud \theta\le C\delta .
\end{align*}

\textsf{Estimate of $\mathcal {K}_2^i$:} 
Similar to the estimation of $\mathcal {J}_2^i$, we denote 
 \begin{align*}
\left|\mathcal K^i_2\right|\le&C
\int_0^1\int^{t_i}_{t_{i-1}}\int_{0}^{1} \int_{t_{i-1}}^r\int_0^1\|D_{\theta,\xi}\la\mathcal{D}\varphi_T^x(t_i,Y_i^\tau),G_{t_i-r}(\cdot,y)\ra \|_{k,p}G_{r-\theta}(y,\xi)\ud\xi\ud \theta\ud y\ud r\ud\tau\\
=&:C\int_0^1B_i(T,x;k,p,\tau)\ud\tau,
\end{align*}
and then  prove that there exists $C=C(T,k,p,\nu)$ such that for any $\tau\in(0,1)$, $x\in(0,1)$ and $t\in(t_i,T]$,
\begin{equation*}\label{KTx1}
B_i(t,x;k,p,\nu,\tau)\le C\delta^2,
\end{equation*}
whose proof is similar to that of \eqref{KTx}. Indeed,
taking the Malliavin derivative $D_{\theta,\xi}$ on both sides of \eqref{Dvar}, 
it follows from 
\begin{align*}
&\int_{t_i}^{t}\int_{0}^{1}\int^{t_i}_{t_{i-1}}\int_{0}^{1} \int_{t_{i-1}}^r\int_0^1G_{t-\theta_1}(x,z)G_{\theta_1-\theta}(z,\xi)G_{\theta_1-r}(z,y)G_{r-\theta}(y,\xi)\ud\xi\ud\theta \ud y\ud r\ud z\ud \theta_1\\
=&\int_{t_i}^{t}\int_{0}^{1}\int^{t_i}_{t_{i-1}} \int_{t_{i-1}}^r\int_0^1G_{t-\theta_1}(x,z)G^2_{\theta_1-\theta}(z,\xi)\ud\xi\ud\theta\ud r\ud z\ud \theta_1\\
\le&C\int_{t_i}^{t}\int^{t_i}_{t_{i-1}}\int_{t_{i-1}}^r\frac{1}{\sqrt{\theta_1-\theta}}\ud\theta\ud r\ud \theta_1\le C\int_{t_i}^{t}\frac{1}{\sqrt{\theta_1-t_i}}\ud\theta_1\int^{t_i}_{t_{i-1}}\int_{t_{i-1}}^r\ud r\ud \theta\le C\delta^2,
\end{align*}
and Lemmas \ref{Ykp0} and \ref{MD0} that
\begin{align*}
B_i(t,x;0,p,\nu,\tau)\le\int_{t_i}^{t}\int_{0}^{1}G_{t-\theta_1}(x,z)B_i(\theta_1,z;0,p,\tau)
\ud z\ud\theta_1+C\delta^2.
\end{align*}

\textsf{Estimate of $\mathcal {K}_3^i$:}
By Corollary \ref{DZ}, for any $t_{i-1}<\theta< r\le t_i$ and $\beta\in(0,1)$,
\begin{align}\label{Dz}
\|D_{\theta,\xi}Z_i^\beta(r,y)\|_{k,p}\le CG_{r-\theta}(y,\xi).
\end{align}
By applying Lemma \ref{G11}, Lemma \ref{MD0}, there exist some positive constants $k,\,p$ such that
\begin{align}\nonumber
\sum_{i=1}^N\left|\mathcal K^i_3\right|\le&C\sum_{i=1}^N
\int_0^1\int_0^1\int^{t_i}_{t_{i-1}}\int_{0}^{1} \int_{t_{i-1}}^r\int_0^1\|\la\mathcal{D}\varphi_T^x(t_i,Y_i^\tau),G_{t_i-r}(\cdot,y)\ra b^{\prime\prime}(Z_i^\beta(r,y))\\\nonumber
&\qquad\qquad D_{\theta,\xi}b^\prime(Z_i^\beta(r,y))\|_{k,p}
G_{r-\theta}(y,\xi)\ud\xi\ud \theta\ud y\ud r\ud\beta\ud\tau\\\label{K3iN}
\le&C\sum_{i=1}^N\int^{t_i}_{t_{i-1}}\int_{0}^{1} \int_{t_{i-1}}^r\int_0^1G_{T-r}(x,y)G^2_{r-\theta}(y,\xi)\ud\xi\ud \theta\ud y\ud r\le C\delta^{\frac{1}{2}}.
\end{align}
Gathering all above estimates, we complete the proof.
\qed

\begin{rems}\label{rem1}
(1) If $b(u)=b_1u+c$ is an affine function, then $b^\prime(Z_i^\beta(r,y))\equiv0$ and thereby 
$\mathcal J^i_3=\mathcal K^i_3=0,\,i=1,\cdots,N.$ In this case, by gathering the estimates on $\{I_i^j\}_{i=1,\cdots,N,\,j=1,\cdots,5}$ in the proof of Theorem \ref{err}, we have,  instead of \eqref{order}, that
 \begin{equation*}
\left|\E[f(u^\delta(T,x))]- \E\left[f(u(T,x))\right]\right|\le C(T,b,\sigma,\|u_0\|_E,\nu)\delta^\nu,
 \end{equation*}
  for every $\nu\in(\frac{1}{2},1)$.
  
(2) With the same idea of Taylor expansion and error decomposition technique as  the proof of the above theorem, we may have the following result on weak convergence as well:

 Let $f:\R\rightarrow\R$ be smooth with bounded derivatives and $0<\delta \le 1$.
Assume that $b\in\mathcal C_\mathbf b^2$. Then there exists some positive constant $C=C(T,b,\sigma,\|u_0\|_E,f)$ such that 
 \begin{equation}\label{rl}
\left|\E[f(u^\delta(T,x))]- \E\left[f(u(T,x))\right]\right|\le C\delta^{\frac{1}{2}},\forall\,x\in(0,1).
 \end{equation}
 
Note that we don't need Lemma \ref{G11} since the generic constant $C$ may depend on $f$ here. As a result, the requirements of the perturbation parameter $\delta$ and the regularity of $b$ are not as strict as those of Theorem \ref{err}.
\end{rems}

\subsection{Analysis with small drift}
 It is obvious that if $b=0$,  the solution $u^\delta(T,x)$ of Eq. \eqref{SAEE} is exactly the exact solution $u(T,x)$ of Eq. \eqref{SHE}. In this part,
we consider the weak convergence with small drift $b$,  that is $b(u)=\varepsilon \widetilde b(u)$ for small $0<\varepsilon<1$ and $\widetilde b\in \mathcal C_\mathbf{b}^3$ is not affine. In this case, we observe that $\mathcal J^i_3$ and $\mathcal K^i_3$ are  bounded by $C(T,\sigma,\widetilde b) \varepsilon\delta^{\frac{1}{2}}$.
By borrowing the notation from the proof of Theorem \ref{err},  for any $\nu\in(\frac{1}{2},1)$, there exists  $C=C(\nu,T,\sigma,\widetilde b,\|u_0\|_E,f)$ such that
\begin{equation}\label{L3}
\left|\sum_{j=1}^5\sum_{i=1}^N\E\left[\mathcal I_i^j\right]-\sum_{i=1}^N(\mathcal J^i_3+\mathcal K^i_3)\right|\le C\varepsilon\delta ^{\nu},
\end{equation}
and thereby for $0<\delta \le1$, 
\begin{align}\label{small d1}
\left|\E[f(u^\delta(T,x))]- \E\left[f(u(T,x))\right]\right|\le C\varepsilon\delta ^{\nu}+|\sum_{i=1}^N\mathcal J^i_3|+|\sum_{i=1}^N\mathcal K^i_3|
\le C\varepsilon\delta^{\frac{1}{2}}.
 \end{align} 
The main result of this part is the following proposition.
\begin{pro}
Let $b(u)=\epsilon \widetilde b(u)$ for small $0<\epsilon<1$ and $\widetilde b\in \mathcal C_\mathbf{b}^3$, $f:\R\rightarrow\R$ be smooth with bounded derivatives and $0<\delta \le 1$. Then for any $\nu\in(\frac{3}{4},1)$,
there exists some positive constant $C=C(T,\widetilde b,\sigma,\|u_0\|_E,f,\nu)$ such that 
 \begin{equation}\label{small d2}
\left|\E[f(u^\delta(T,x))]- \E\left[f(u(T,x))\right]\right|\le C\varepsilon\delta^{\nu-\frac{1}{4}}+C\varepsilon^2\delta^{\frac{1}{2}},\forall\,x\in(0,1).
 \end{equation}
\end{pro}

\begin{proof}
For the sake of simplicity, denote
\begin{align*}
I_i(\beta,r,y):=&\int_0^{t_{i-1}}\int_0^1D_{\theta,\xi}Z_i^\beta(r,y)\{G_{r-\theta}(y,\xi)-G_{t_{i-1}-\theta}(y,\xi)\}\sigma\ud\xi\ud\theta\\
&+\int_{t_{i-1}}^r\int_0^1D_{\theta,\xi}Z_i^\beta(r,y)G_{r-\theta}(y,\xi)\sigma\ud\xi\ud \theta.
\end{align*}
We recall that $\mathcal{J}_3^i+\mathcal{K}_3^i$ is equal to
\begin{align*}
\int_0^1\int_0^1\int^{t_i}_{t_{i-1}}\int_{0}^{1}\E\Big[f^\prime(\varphi_T^x(t_i,Y_i^\tau))\la\mathcal{D}\varphi_T^x(t_i,Y_i^\tau),G_{t_i-r}(\cdot,y)\ra
b^{\prime\prime}(Z_i^\beta(r,y))I_i(\beta,r,y)\Big]\ud y\ud r\ud\beta\ud\tau
\end{align*}
and rewrite $\mathcal{J}_3^i+\mathcal{K}_3^i=\mathcal L^1_i+\mathcal L^2_i$ with
\begin{align*}
\mathcal L^1_i:=&\int_0^1\int^{t_i}_{t_{i-1}}\int_{0}^{1}\E\Big[f^\prime(\varphi_T^x(t_i,Y_i^\tau))\la\mathcal{D}\varphi_T^x(t_i,Y_i^\tau),G_{t_i-r}(\cdot,y)\ra
b^{\prime\prime}(\Phi_{t_{i-1}}^y(0,u_0))\\
&\qquad\qquad\qquad\qquad\int_0^1I_i(\beta,r,y)\ud\beta\Big]\ud y\ud r\ud\tau,\\
\mathcal L^2_i:=&\int_0^1\int_0^1\int^{t_i}_{t_{i-1}}\int_{0}^{1}\E\Big[f^\prime(\varphi_T^x(t_i,Y_i^\tau))\la\mathcal{D}\varphi_T^x(t_i,Y_i^\tau),G_{t_i-r}(\cdot,y)\ra
\Delta_i(\beta,r,y)\\
&\qquad\qquad\qquad\qquad I_i(\beta,r,y)\Big]\ud y\ud r\ud\beta\ud\tau,
\end{align*}
where 
\begin{align*}
\Delta_i(\beta,r,y):=&b^{\prime\prime}(Z_i^\beta(r,y))-b^{\prime\prime}(\Phi_{t_{i-1}}^y(0,u_0))\\=&\int_0^1b^{\prime\prime\prime}\left((1-\zeta)\Phi_{t_{i-1}}^y(0,u_0)+\zeta Z_i^\beta(r,y)\right)\ud\zeta \left(Z_i^\beta(r,y)-\Phi_{t_{i-1}}^y(0,u_0)\right).
\end{align*}
First, we  proceed to estimate $I_i(\beta,r,y)$ as follows. Notice that
\begin{align}\nonumber\label{lowN}
&\int_{0}^{1} \int_0^{t_{i-1}}D_{\theta,\xi}\Phi^y_{t_{i-1}}(0,u_0)\{G_{r-\theta}(y,\xi)-G_{t_{i-1}-\theta}(y,\xi)\}\sigma\ud\theta\ud\xi\\\nonumber
=&\sum_{j=1}^{i-2}\int_{0}^{1} \int_0^{t_{i-1}}\int_{t_j}^{t_{j+1}}\int_{0}^{1}G_{t_{i-1}-s}(y,z)
b^\prime(\Phi^z_{t_j}(0,u_0))D_{\theta,\xi}\Phi^z_{t_j}(0,u_0)\{G_{r-\theta}(y,\xi)-G_{t_{i-1}-\theta}(y,\xi)\}\\\nonumber
&\sigma\ud z\ud s\ud\theta\ud\xi+\int_{0}^{1} \int_0^{t_{i-1}}G_{t_{i-1}-\theta}(y,\xi)\sigma^2\{G_{r-\theta}(y,\xi)-G_{t_{i-1}-\theta}(y,\xi)\}\ud\theta\ud\xi\\
=&:\sum_{j=1}^{i-2}B^j_i+\int_{0}^{1} \int_0^{t_{i-1}}G_{t_{i-1}-\theta}(y,\xi)\sigma^2\{G_{r-\theta}(y,\xi)-G_{t_{i-1}-\theta}(y,\xi)\}\ud\theta\ud\xi.
\end{align} 
For any $j$ with $1\le j\le i-2$, we denote $A_i^1:=\sum_{j=1}^{i-2}B^j_i$ and decompose $B^j_i=B_{i,1}^j+B_{i,2}^j$ with
\begin{align*}
B_{i,1}^j:=&\int_{0}^{1} \int_0^{t_{j-1}}\int_{t_j}^{t_{j+1}}\int_{0}^{1}G_{t_{i-1}-s}(y,z)
b^\prime(\Phi^z_{t_j}(0,u_0))D_{\theta,\xi}\Phi^z_{t_j}(0,u_0)\\
&\qquad\qquad\qquad\qquad\{G_{r-\theta}(y,\xi)-G_{t_{i-1}-\theta}(y,\xi)\}\ud z\ud s\ud\theta\ud\xi,\\
B_{i,2}^j:=&\int_{0}^{1} \int_{t_{j-1}}^{t_j}\int_{t_j}^{t_{j+1}}\int_{0}^{1}G_{t_{i-1}-s}(y,z)
b^\prime(\Phi^z_{t_j}(0,u_0))D_{\theta,\xi}\Phi^z_{t_j}(0,u_0)\\
&\qquad\qquad\qquad\qquad\{G_{r-\theta}(y,\xi)-G_{t_{i-1}-\theta}(y,\xi)\}\ud z\ud s\ud\theta\ud\xi,
\end{align*}
because of $D_{\theta,\xi}\Phi^z_{t_j}(0,u_0)=0$, if $\theta>t_j$.
By choosing $\nu\in(\frac{1}{2},1)$, it leads to 
\begin{align*}
\|B_{i,1}^j\|_{2}
\le&C|b|_1\int_{0}^{1} \int_0^{t_{j-1}}\int_{t_j}^{t_{j+1}}\int_{0}^{1}
G_{t_{i-1}-s}(y,z)(t_j-\theta)^{-\frac{1}{2}}|G_{r-\theta}(y,\xi)-G_{t_{i-1}-\theta}(y,\xi)|\ud z\ud s\ud\theta\ud\xi\\
\le&C\delta ^{\nu}|b|_1 \int_0^{t_{j-1}}\int_{t_j}^{t_{j+1}}\int_{0}^{1}
G_{t_{i-1}-s}(y,z)(t_j-\theta)^{-\frac{1}{2}}(t_{i-1}-\theta)^{-\nu}\ud z\ud s\ud\theta\\
\le&C|b|_1\delta^{1+\nu} \int_0^{t_{j-1}}
(t_j-\theta)^{-\frac{1}{2}}(t_{i-1}-\theta)^{-\nu}\ud\theta
\le C|b|_1\delta^{1+\nu} \int_0^{t_{j-1}}
(t_j-\theta)^{-\frac{1}{2}-\nu}\ud\theta\le C|b|_1\delta^{\frac{3}{2}}
\end{align*}
and
\begin{align*}
\|B_{i,2}^j\|_2
\le&C\int_{t_{j-1}}^{t_{j}}\int_{t_j}^{t_{j+1}}\int_{0}^{1}
G_{t_{i-1}-s}(y,z)|b|_1(t_j-\theta)^{-\frac{1}{2}}\int_{0}^{1} |G_{r-\theta}(y,\xi)-G_{t_{i-1}-\theta}(y,\xi)|\ud\xi\ud z\ud s\ud\theta\\
\le&C\int_{t_{j-1}}^{t_{j}}\int_{t_j}^{t_{j+1}}\int_{0}^{1}
G_{t_{i-1}-s}(y,z)|b|_1(t_j-\theta)^{-\frac{1}{2}}\ud z\ud s\ud\theta\\
\le&C|b|_1\delta  \int_{t_{j-1}}^{t_{j}}
(t_j-\theta)^{-\frac{1}{2}}\ud\theta\le C |b|_1\delta^{\frac{3}{2}}
\end{align*}
for some $C=C(T,\sigma,\widetilde b,\nu)$.
 Hence,  it follows that   
\begin{align}\label{Ai1}
\|\sum_{j=1}^{i-2}B^j_i\|_2=\|\sum_{j=1}^{i-2}(B_{i,1}^j+B_{i,2}^j)\|_2\le C(T,\sigma,\widetilde b)|b|_1\delta^{\frac{1}{2}}.
\end{align}
In the same way, by noticing that for $0<\theta<t_{i-1}<r<t_i$,
 \begin{align*}
&D_{\theta,\xi}\varphi_{r}^y(t_{i-1},\Phi_{t_{i-1}}(0,u_0))\\
=&G_{r-\theta}(y,z)\sigma+\sum_{j=1}^{i-2}\int_{t_j}^{t_{j+1}}\int_0^1G_{r-s}(y,z)b^\prime(\Phi_{t_j}^z(0,u_0))D_{\theta,\xi}\Phi_{t_j}^z(0,u_0)\ud z\ud s\\
&+\int_{t_{i-1}}^r\int_0^1G_{r-s}(y,z)b^\prime(\varphi_{s}^z(t_{i-1},\Phi_{t_{i-1}}(0,u_0)))D_{\theta,\xi}\varphi_{s}^z(t_{i-1},\Phi_{t_{i-1}}(0,u_0))\ud z\ud s
\end{align*}
and the estimate
 \begin{align*}
&\Big|\int_{0}^{1} \int_0^{t_{i-1}}\int_{t_{i-1}}^r\int_0^1G_{r-s}(y,z)b^\prime(\varphi_{s}^z(t_{i-1},\Phi_{t_{i-1}}(0,u_0)))D_{\theta,\xi}\varphi_{s}^z(t_{i-1},\Phi_{t_{i-1}}(0,u_0))\\
&\{G_{r-\theta}(y,\xi)-G_{t_{i-1}-\theta}(y,\xi)\}\sigma\ud z\ud s\ud\theta\ud\xi\Big|\\
\le&C(T,\sigma,\widetilde b)|b|_1\int_{0}^{1} \int_0^{t_{i-1}}\int_{t_{i-1}}^r\int_0^1G_{r-s}(y,z)G_{s-\theta}(z,\xi)
|G_{r-\theta}(y,\xi)-G_{t_{i-1}-\theta}(y,\xi)|\ud z\ud s\ud\theta\ud\xi\\
\le&C(T,\sigma,\widetilde b)|b|_1 \int_0^{t_{i-1}}\int_{t_{i-1}}^r(r-\theta)^{-\frac{1}{2}}
\int_{0}^{1}G_{r-\theta}(y,\xi)+G_{t_{i-1}-\theta}(y,\xi)\ud\xi\ud s\ud\theta\\
\le& C(T,\sigma,\widetilde b)|b|_1\delta^{\frac{1}{2}},
\end{align*}
we can derive that for some $A_i^2$,
\begin{align}\nonumber
&\int_{0}^{1} \int_0^{t_{i-1}}D_{\theta,\xi}\varphi_{r}^y(t_{i-1},\Phi_{t_{i-1}}(0,u_0))\{G_{r-\theta}(y,\xi)-G_{t_{i-1}-\theta}(y,\xi)\}\sigma\ud\theta\ud\xi\\\label{lowE}
=&\int_{0}^{1} \int_0^{t_{i-1}}G_{r-\theta}(y,\xi)\sigma^2\{G_{r-\theta}(y,\xi)-G_{t_{i-1}-\theta}(y,\xi)\}\ud\theta\ud\xi+A_i^2
\end{align} 
with 
\begin{align}\label{Ai2}
\|A_i^2\|_2\le C(\widetilde b,T,\sigma)|b|_1\delta^{\frac{1}{2}}.
\end{align}
By the semigroup property of $G$ and \eqref{GD}, for any $r_1\in(t_{i-1},t_i]$,
\begin{align*}
&\int_{0}^{1} \int_0^{t_{i-1}}G_{r_1-\theta}(y,\xi)\sigma^2\{G_{r-\theta}(y,\xi)-G_{t_{i-1}-\theta}(y,\xi)\}\ud\theta\ud\xi\\=&\int_0^{t_{i-1}}\sigma^2\{G_{r_1-\theta+r-\theta}(y,y)-G_{r_1-\theta+t_{i-1}-\theta}(y,y)\}\ud\theta\\
=&2\sigma^2 \sum_{k=0}^{\infty}\int_0^{t_{i-1}}e^{-k^{2} \pi^{2} (r_1-\theta+r-\theta)}-e^{-k^{2} \pi^{2} (r_1-\theta+t_{i-1}-\theta)}\ud\theta\cos^2 (k \pi y).
 \end{align*}
Then \eqref{lowN} and \eqref{lowE}, and the definition of $Z_i^\beta(r,y)$ give that
\begin{align*}
&\int_{0}^{1} \int_0^{t_{i-1}}D_{\theta,\xi}Z_i^\beta(r,y)\{G_{r-\theta}(y,\xi)-G_{t_{i-1}-\theta}(y,\xi)\}\sigma\ud\theta\ud\xi\\
= &2\beta\sigma^2 \sum_{k=1}^{\infty}\int_0^{t_{i-1}}e^{-k^{2} \pi^{2} (r-\theta+r-\theta)}-e^{-k^{2} \pi^{2} (r-\theta+t_{i-1}-\theta)}\ud\theta\cos^2 (k \pi y)+\beta A_i^2\\
&+2(1-\beta)\sigma^2 \sum_{k=1}^{\infty}\int_0^{t_{i-1}}e^{-k^{2} \pi^{2} (t_{i-1}-\theta+r-\theta)}-e^{-k^{2} \pi^{2} (t_{i-1}-\theta+t_{i-1}-\theta)}\ud\theta\cos^2 (k \pi y)+(1-\beta) A_i^1.
\end{align*}
Since for any $\theta\in(t_{i-1},r),\,D_{\theta,\xi}Z_i^\beta(r,y)=\beta\sigma G_{r-\theta}(y,\xi)$, we have that for $\beta\in(0,1)$, 
  \begin{align*}
 & \int_{t_{i-1}}^r\int_0^1D_{\theta,\xi}Z_i^\beta(r,y)G_{r-\theta}(y,\xi)\sigma\ud\xi\ud \theta=
\beta\int_{0}^{1} \int_{t_{i-1}}^r\sigma^2G^2_{r-\theta}(y,\xi)\ud\theta\ud\xi\\
=&\beta\int_{t_{i-1}}^r\sigma^2G_{2(r-\theta)}(y,y)\ud\theta=2\beta\sigma^2 \sum_{k=0}^{\infty}\int_{t_{i-1}}^re^{-k^{2} \pi^{2} (r-\theta+r-\theta)}\ud\theta\cos^2 (k \pi y)\\=&\beta\sigma^2 \sum_{k=1}^{\infty}\frac{1}{k^2\pi^2}\left(1-e^{-2k^2\pi^2(r-t_{i-1})}\right)
\cos^2 (k \pi y)+2\beta\sigma^2(r-t_{i-1}).
 \end{align*}
Summarizing the above calculations, we obtain 
 \begin{align}\nonumber\label{Iib}
\int_0^1I_i(\beta,r,y)\ud\beta
= &\sigma^2 \sum_{k=1}^{\infty}\int_0^{t_{i-1}}e^{-k^{2} \pi^{2} (r-\theta+r-\theta)}-e^{-k^{2} \pi^{2} (t_{i-1}-\theta+t_{i-1}-\theta)}\ud\theta\cos^2 (k \pi y)\\\nonumber
&+\frac{\sigma^2}{2} \sum_{k=1}^{\infty}\frac{1}{k^2\pi^2}\left(1-e^{-2k^2\pi^2(r-t_{i-1})}\right)
\cos^2 (k \pi y)+\sigma^2(r-t_{i-1})+\frac{1}{2}(A_i^2+ A_i^1)\\\nonumber
=&\frac{\sigma^2}{2} \sum_{k=1}^{\infty}\frac{1}{k^2\pi^2}\left(e^{-2k^2\pi^2t_{i-1}}-e^{-2k^2\pi^2r}\right)
\cos^2 (k \pi y)+\sigma^2(r-t_{i-1})+\frac{1}{2}(A_i^2+ A_i^1)\\
=&\sigma^2 \sum_{k=1}^{\infty}\int_{t_{i-1}}^re^{-2k^2\pi^2s}\ud s
\cos^2 (k \pi y)+\sigma^2(r-t_{i-1})+\frac{1}{2}(A_i^2+ A_i^1).
\end{align}
From \eqref{Ai1}, \eqref{Ai2} and the inequality $\sum_{k=1}^{\infty}e^{-2k^2\pi^2s}\le (8\pi s)^{-\frac{1}{2}}$, we have that for some $C=C(|f|_1,T,\sigma,\widetilde b)$,
\begin{align}\nonumber
&\left|\sum_{i=1}^N\mathcal L^1_i\right|\le C|b|_2\sum_{i=1}^N\int^{t_i}_{t_{i-1}}\int_{0}^{1}G_{T-r}(x,y)
\left\|\int_0^1I_i(\beta,r,y)\ud\beta\right\|_2\ud y\ud r\\\nonumber
\le& C|b|_2\sum_{i=1}^N\int^{t_i}_{t_{i-1}}
\int_{t_{i-1}}^rs^{-\frac{1}{2}}\ud s\ud r+C|b|_2\left(|b|_1\delta^{\frac{1}{2}}+\delta \right)\\
\label{L1}
\le&C|b|_2\left(|b|_1\delta^{\frac{1}{2}}+\delta \right).
\end{align}
Now we estimate $\mathcal L^2_i$. By \eqref{Zry} and the fact that the exact flow $\varphi$ associated to Eq. \eqref{SHE} is almost $1/4$-H\"older continuous in time (see e.g. \cite[Proposition 2.4.3]{DN06}),  for $r\in(t_{i-1},t_i]$ and $\nu\in(\frac{3}{4},1)$, we have
\begin{align*}
\|\Delta_i(\beta,r,y)\|_2\le|b|_3\left\|\varphi_r^y(t_{i-1},\Phi_{t_{i-1}}(0,u_0))-\Phi_{t_{i-1}}^y(0,u_0)\right\|_2\le C(\nu,T,\widetilde b,\sigma,\|u_0\|_E)|b|_3\delta^{\nu-\frac{3}{4}}.
\end{align*}
By replacing $b^{\prime\prime}(Z_i^\beta(r,y))$ by $\Delta_i(\beta,r,y)$ in \eqref{J3iN} and \eqref{K3iN}, it follows that for any $\nu\in(\frac{3}{4},1)$,
\begin{align}\label{L2}
&\left|\sum_{i=1}^N\mathcal L^2_i\right|
\le C(|f|_1,T,\sigma,\nu,\widetilde b)|b|_3\delta^{\nu-\frac{1}{4}}.
\end{align}
Finally, by noticing that $|b|_i=\epsilon|\widetilde b|_i,\,i=1,2,3$, and combining \eqref{L3}, \eqref{L1} and \eqref{L2},  we complete the proof.
\end{proof}

\section{Convergence of Density Approximations}\label{S5}
In this section, we focus on the convergence of density approximations and  the logarithmic  asymptotic behavior of the densities. 


\subsection{Convergence of Densities}
This part investigates the convergence of density approximations for Eq. \eqref{SHE} in both uniform convergence  topology and total variation distance.  We would like to mention that there already exist some convergence results of density approximations for stochastic ordinary differential equations (see e.g. \cite{BT96,CHS19} and references therein), but few results on stochastic partial differential equations.

From Theorem \ref{smooth} and \cite[Lemma 2.1.7]{DN06}, it follows that
$$\E[\bm{\delta}_{\z}(u^\delta(T,x))]=q_{T,x}^\delta(\z)~ \big(\text{resp.} ~\E\left[\bm{\delta}_{\z}\left(u(T,x)\right)\right]=q_{T,x}(\z)\big)$$ is the density of $u^\delta(T,x)$ (resp. $u(T,x)$) at $\z\in\R$. On the basis of Theorem \ref{err},
we show that the density $q_{T,x}^\delta$ converges to the density $q_{T,x}$ in the uniformly convergence  topology, and  the convergence order coincides with the weak convergence order. For this purpose, we begin with recalling the fact: If a random variable $\mathbf F$ has a smooth density $q$, then
\begin{equation}\label{gnzF}
q(\z)=\lim_{n\rightarrow \infty}\int_{\R}g_{n^{-1}}(\z-\xi)q(\xi)\ud \xi=\lim_{n\rightarrow \infty}\mathbb E[g_{n^{-1}}(\z-\mathbf F)],
\end{equation}
where $g_{n^{-1}}$ is defined by \eqref{Gaussian}. 
%
Now we prove Theorem \ref{dstyt}. 

\textit{Proof of Theorem \ref{dstyt}}:
Let $\z\in\R$ and integer $n\ge1$ be arbitrarily fixed. We
take $f(\y)=g_{n^{-1}}(\y-\z)$ in Theorem \ref{err}. Then $F(\y)=\int_{-\infty}^\y g_{n^{-1}}(\y_1-\z)\ud \y_1$ satisfies $0\le F\le 1$, hence there exists $C=C(T,b,\sigma,\|u_0\|_E,\nu)$ independent of $\z$, $n$ and $x$ such that
\begin{equation*}
\left|\E[g_{n^{-1}}(u^\delta(T,x)-\z)]- \E\left[g_{n^{-1}}(u(T,x)-\z)\right]\right|\le C\delta^{\frac{1}{2}}.
 \end{equation*}
Putting $n\rightarrow\infty$ in the above inequality, then the desired result \eqref{error1} follows from  Theorem \ref{smooth}, the non-degeneracy of $u(T,x)$ (see e.g. \cite[Section 4]{MD08}) and \eqref{gnzF}.
\qed

When $b$ is affine, it can be seen from Remarks \ref{rem1} that the convergence order of density approximations  can be improved to be nearly $1$, which can also be proved based on the strong convergence order; see the following example.
\begin{exam}
We discuss the affine case: $b(u)=b_1u+c$. 
 On the one hand, 
\begin{align*}
u(T,x)=&\int_0^1e^{b_1T}G_{T}\left(x,y\right)u_0(y)\ud y+\int_{0}^{T}\int_{0}^{1}G_{T-s}(x,y)e^{b_1(T-s)}c\ud y\ud s\\
 &+\int_{0}^{T}\int_{0}^{1}G_{T-s}(x,y)e^{b_1(T-s)}\sigma W(\ud s,\ud y)
\end{align*}
indicates that $u(T,x)$ is a Gaussian random variable with mean $m_1:=\int_0^1e^{b_1T}G_{T}\left(x,y\right)u_0(y)\ud y+\frac{c(1-e^{b_1T})}{b_1}$ and variance $\sigma_1:=\int_{0}^{T}\int_{0}^{1}G^2_{T-s}(x,y)e^{2b_1(T-s)}\sigma^2 \ud y\ud s.$ 
On the other hand, 
we use a version of Clark-Ocone formula for two parameter processes to obtain
 \begin{align*}
u^\delta(T,x)=\mathbb{E}[u^\delta(T,x)]+ \int_0^T\int_0^1 \mathbb{E}[D_{r, y}u^\delta(T,x)| \mathcal{F}_{r}] W(\ud r,\ud y),
\end{align*}
where
\begin{align*}
D_{r,y}u^\delta(T,x)=\sum_{j=0}^{N-1}\int_{t_j}^{t_{j+1}}\int_0^1G_{T-s}(x,y)b_1D_{r,y}u^\delta(t_j,y)\ud y\ud s+G_{T-r}(x,y)\sigma
\end{align*}
is independent of $\omega$. Therefore,
for any fixed $x\in(0,1)$, 
$u^\delta(T,x)$ is a Gaussian random variable with mean $m_2:=\E[u^\delta(T,x)]$ and variance $\sigma_2:=\E[u^\delta(T,x)^2]-(\E[u^\delta(T,x)])^2$.

Although it is not easy to give the explicit expressions of $m_2$ and $\sigma_2$, 
the  convergence order of density approximations of Eq. \eqref{SHE} with $b(u)=b_1u+c$ can be obtained by the strong convergence order.  In fact,
 we observe that
\begin{align*}
&|m_1-m_2|\le\|u^\delta(T,x)-u(T,x)\|_1,\\
&|\sigma_1-\sigma_2|\le2\|u^\delta(T,x)-u(T,x)\|_2(\|u^\delta(T,x)\|_2+\|u(T,x)\|_2),
\end{align*}
 and  that  by the mean value theorem
\begin{align*}
&\left|g_{\sigma_1}(\z-m_1)-g_{\sigma_2}(\z-m_2)\right|\\
\le&\left|\frac{\partial}{\partial m}g_{\sigma_1}(\z-(m_1+\theta(m_2-m_1)))\right||m_1-m_2|+\left|\frac{\partial}{\partial \sigma}g_{\sigma_1+\theta(\sigma_2-\sigma_1)}(\z-m_2)\right||\sigma_1-\sigma_2|\\
\le&C(\sigma_1,\sigma_2)\|u^\delta(T,x)-u(T,x)\|_2\left(\|u^\delta(T,x)\|_2+\left\|u(T,x)\right\|_2\right)\\
\le& C\|u^\delta(T,x)-u(T,x)\|_2,
\end{align*}
where $g$ is defined by \eqref{Gaussian} and $C$ is independent of $x\in(0,1),$ $z\in\R$ and $\theta\in(0,1)$. Further, it can be shown that for any $\nu\in(0,1)$, there exists $C=C(T,\nu,b,\sigma)$ such that (see e.g. \cite{AP08}),
\begin{align*}
\sup_{x\in(0,1)}\|u^\delta(T,x)-u(T,x)\|_2\le C\delta ^{\nu}.
\end{align*}
As a result, for any $\nu\in(0,1)$ and $x\in(0,1)$, there exists $C=C(T,\nu,b_1,c,\sigma)$ such that
 \begin{equation*}
 \sup_{\z\in\R}|q_{T,x}^\delta(\z)-q_{T,x}(\z)|\le C\delta ^{\nu}.
 \end{equation*}

 \end{exam}

Recall that the total variation distance of  probability measures $\mu$ and $\nu$ on a $\sigma$-algebra $\Sigma$ is defined by 
\begin{equation}\label{TVD}
d_{TV}(\mu,\nu)=2\sup\{|\mu(A)-\nu(A)|:A\in\Sigma \}.
\end{equation}
Let $\{B_t\}_{t\ge 0}$ be a standard $1$-dimensional Brownian motion on $(\Omega,\mathscr F,\mathbb P)$.
It is shown in \cite[Theorem 2.6]{BT96} that the numerical approximation $Y_N$ of the Euler-Maruyama scheme of elliptic stochastic differential equations $\ud Y (t) = f(Y (t))\ud t + \ud B_t,\,t\in[0,T]$, has weak convergence order $1$ even for bounded continuous test functions. This implies that
$$\lim_{\delta \rightarrow 0}d_{TV}\left(Y(T)\circ\mathbb P^{-1},Y_N\circ\mathbb P^{-1}\right)= 0.$$
For infinite dimensional case, we can derive a similar result that for Eqs. \eqref{SHE} and \eqref{SAEE} and for any fixed $x\in(0,1)$,
$$\lim_{\delta \rightarrow 0}d_{TV}\left(u(T,x)\circ\mathbb P^{-1},u^\delta(T,x)\circ\mathbb P^{-1}\right)=0.$$


In fact, because $u(T,x)$ and $u^\delta(T,x)$ have smooth densities $q_{T,x}$ and $q_{T,x}^\delta$, respectively, it is readily to verify that the set $A=\{\z:q_{T,x}(\z)>q_{T,x}^\delta(\z)\}$ attains  the supremum of $\sup\{|\mathbb P(u(T,x)\in A)-\mathbb P(u^\delta(T,x)\in A)|:A\in\mathscr B(\R) \}$, which leads to 
 $$d_{TV}\left(u(T,x)\circ\mathbb P^{-1},u^\delta(T,x)\circ\mathbb P^{-1}\right)=\int_\R|q_{T,x}^\delta(\z)-q_{T,x}(\z)|\ud \z.$$
For any $\eta\in(0,\frac{1}{2})$ and $\delta>0$,
 it follows from \eqref{error1} that
  \begin{equation*}
 \int_{-\delta^{-\eta}}^{\delta^{-\eta}}|q_{T,x}^\delta(\z)-q_{T,x}(\z)|\ud \z\le 2C\delta^{-\eta}\delta^{\frac{1}{2}}\le 2C\delta^{\frac{1}{2}-\eta}.
 \end{equation*}
Accordingly, we obtain
\begin{align*}
\int_\R|q_{T,x}^\delta(\z)-q_{T,x}(\z)|\ud \z\le&2C\delta^{\frac{1}{2}-\eta}+
\int_{-\infty}^{-\delta^{-\eta}}|q_{T,x}^\delta(\z)|\ud \z+\int_{\delta^{-\eta}}^{\infty}|q_{T,x}^\delta(\z)|\ud \z\\&+\int_{-\infty}^{-\delta^{-\eta}}|q_{T,x}(\z)|\ud \z+\int_{\delta^{-\eta}}^{\infty}|q_{T,x}^\delta(\z)|\ud \z\rightarrow 0,\,\text{as}\,\delta\rightarrow0,
 \end{align*}
since the last four integrals tend to $0$ as $\delta\rightarrow 0$ thanks to $\int_\R|q_{T,x}^\delta(\y)|\ud \y=\int_\R\left|q_{T,x}(\y)\right|\ud \y=1$.
\subsection{Logarithmic of asymptotic property}
In this part, we present the logarithmic asymptotic property of the density of the exact solution of Eq. \eqref{SHE}, which turn out to be preserved by Eq. \eqref{SAEE} exactly.
For this end, we begin with briefly recalling the Nourdin and Viens's result \eqref{rho0} on dominating the density of a general centered random variable $Z$ from above and below by means of Malliavin calculus.

%

For $Z\in\D^{1,2}$ with mean zero, define the function $h$ by 
\begin{equation*}\label{gz}
h(\z):=\E[\la DZ,-DL^{-1}Z\ra_{\uH}|Z=\z],\,\forall\,\z\in\R,
\end{equation*}
where $L^{-1}$ is the inverse of infinitesimal generator $L$ of Ornstein-Uhlenbeck semigroup.
If there exist $\sigma_{min},\,\sigma_{max}>0$ such that
\begin{equation*}
\sigma^2_{min}\le h(Z)\le \sigma^2_{max},\,a.s.,
\end{equation*}
then, by \cite[Corollary 3.5]{NV08}, $Z$ has a density $\rho$ satisfying, for almost every $\z\in\R$,
\begin{align}\label{rho0}
\frac{\E|Z|}{2 \sigma_{\min }^{2}} \exp \left(-\frac{\z^{2}}{2 \sigma_{\max }^{2}}\right) 
\le  \rho(\z) \le\frac{E|Z|}{2 \sigma_{\max }^{2}} \exp \left(-\frac{\z^{2}}{2 \sigma_{\min }^{2}}\right).
\end{align}

Suppose that the process $W^\prime=\{W^\prime(h),h\in\uH\}$ is an independent copy of $W$. If there is no confusion caused, $W:(\Omega,\mathscr F,\mathbb P)\rightarrow\R^\uH$ and   $W^\prime:(\Omega^\prime,\mathscr F^\prime,\mathbb P^\prime)\rightarrow\R^\uH$ can be  seen as the canonical mappings associated with the processes $W=\{W(h),h\in\uH\}$ and $W^\prime=\{W^\prime(h),h\in\uH\}$, respectively. 
If $Z\in\mathbb D^{1,2}$, we  write $DZ=\Psi_Z\circ W,$ where $\Psi_Z$ is a measurable mapping from $\R^{\uH}\rightarrow \uH$, determined $\mathbb P\circ W^{-1}$-almost surely (\cite[Section 1.4.1]{DN06}). Further, by \cite[Proposition 3.5]{NV08},
$h(Z)$ can be rewritten as 
\begin{equation*}
h(Z)=\int_0^\infty e^{-\theta}\mathbf E\left[\la \Psi_{Z}\circ W,\Psi_{Z}\circ (e^{-\theta}W+\sqrt{1-e^{-2\theta}}W^\prime)\ra_{\uH}\big|Z\right]\ud \theta,
\end{equation*}
where 
$\mathbf E$ denotes the expectation with respect to $\mathbb P\times\mathbb P^\prime$. By denoting $\boldsymbol{\omega}:=(\omega,\omega^\prime)$ and
\begin{align}\label{DZw}
\widetilde{DZ}(\boldsymbol{\omega}):=\Psi_{Z}\circ \left(e^{-\theta}W(\omega)+\sqrt{1-e^{-2\theta}}W^\prime(\omega^\prime)\right),
\end{align}
we have
$$h(Z)=\int_0^\infty e^{-\theta}\E\left[\E^\prime[\la DZ,\widetilde{DZ}\ra_{\uH}]\Big|Z\right]\ud \theta,$$
where $\E^\prime$ denotes the expectation with respect to $\mathbb P^\prime$ and the explicit dependence of $\widetilde {DZ}$ upon $\theta$ is dropped for simplicity of notation.

Based on the above techniques, we show the following theorem. 
\begin{tho}\label{Asy}
Let $b\in\mathcal C_{\mathbf b}^1$.  Then for any $x\in[0,1]$, $u(t,x)$ admits a density $q_{t,x}$ satisfying that for almost every $\z\in\R$,
\begin{align}\label{asyd}
\lim _{t \rightarrow 0} t^{\frac{1}{2}} \log q_{t,x}(\z)&=-\frac{\sqrt{ 2\pi}}{4\sigma^2}(1+\mathrm{sgn}(x(1-x)))(\z-u_0(x))^{2}.
\end{align}
\end{tho}
\begin{proof}

Without loss of generality, we assume that $\sigma>0$.
From \cite[Proposition 2.4.4]{DN06}, we have $u(t, x)\in\mathbb D^{1,2}$. 
For any fixed $(r,z)\in(0,T)\times[0,1]$, the Malliavin derivative $D_{r, z} u(t, x)$ satisfies 
\begin{equation*}
D_{r, z} u(t, x)=\sigma G_{t-r}(x,z)+\int_r^t\int_0^1b^{\prime}(u(s, y))D_{r, z}u(s, y)\ud y\ud s.
\end{equation*} 
Noticing that $-|b|_1\le b^{\prime}(u(s, y))\le |b|_1$, and by the comparison principle (\cite[Lemma 4]{MD08}), we obtain that, except on a $\mathbb{P}$-null set, for all $(t,x)\in(r,T]\times[0,1]$,
\begin{equation*}
 e^{-|b|_1(t-r)}\sigma G_{t-r}(x, z)\le D_{r, z} u(t, x)\le e^{|b|_1(t-r)}\sigma G_{t-r}(x, z).
\end{equation*}
Due to $u(t,x)\in \mathbb D^{1,2}$,  $Du(t,x)(\omega)=\Psi_{u(t,x)}(W(\omega))$ for some measurable mapping $\Psi_{u(t,x)}$ from $\R^{\uH}$ to $\uH$, $\mathbb P\circ W^{-1}$-a.s. For any $(r,z)\in(0,T)\times(0,1)$, we write $$\Psi^{r,z}_{u(t,x)}(W):=D_{r, z}u(t, x),$$
and then conclude that
 \begin{equation*}
e^{-|b|_1(t-r)}\sigma G_{t-r}(x, z)\le \Psi^{r,z}_{u(t,x)}(W)\le e^{|b|_1(t-r)}\sigma G_{t-r}(x, z),\,\mathbb P-a.s.
\end{equation*}
Substituting $Z$ by $u(t,x)$ in \eqref{DZw}, we denote $\widetilde {Du(t,x)}=\Psi_{u(t,x)}(e^{-\theta}W+\sqrt{1-e^{-2\theta}}W^\prime)$.
Since the process $\mathbf W=\{\mathbf W(h),\,h\in\uH\}$ defined by $$\mathbf W(h)=e^{-\theta}W(h)+\sqrt{1-e^{-2\theta}}W^\prime(h),\,h\in\uH,$$
 is Gaussian on the product probability space $(\Omega\times\Omega^\prime,\mathscr F\bigotimes \mathscr F^\prime,\mathbb P\times\mathbb P^\prime)$, with mean zero and with the same covariance function as $W$(see \cite[Section 1.4.1]{DN06}),
 $\widetilde {Du(t, x)}=\{\widetilde {D_{r,z}u(t, x)},(r,z)\in(0,T]\times(0,1)\}$ has the same distribution as $Du(t,x)=\{{D_{r,z}u(t, x)},(r,z)\in(0,T]\times(0,1)\}$ and hence satisfies,
 except on a $\mathbb{P}\times\mathbb{P}^\prime$-null set,   
 \begin{equation*}
e^{-|b|_1(t-r)}\sigma G_{t-r}(x, z)\le \widetilde {D_{r, z} u(t, x)}\le e^{|b|_1(t-r)}\sigma G_{t-r}(x, z).
\end{equation*}
Putting the above arguments together, for any $t>0$ and $x\in[0,1]$, we obtain
 \begin{align*} 
 h(u(t,x)) 
 &=\int_{0}^{\infty} e^{-\theta} \E\left[\E^{\prime}\left(\int_{0}^{t} \int_{0}^{1} D_{r, z} u(t, x)\widetilde{D_{r, z} u(t, x)}  \ud z \ud r\right)\Big | u(t,x)\right] \ud \theta \\
 &\ge \int_0^t\int_0^1e^{-2|b|_1(t-r)}\sigma^2 G^2_{t-r}(x, z)\ud z \ud r= \int_0^t\int_0^1e^{-2|b|_1 r}\sigma^2 G^2_r(x, z)\ud z \ud r=:\sigma^2_{\min}
 \end{align*}
 and
  \begin{align*} 
 h(u(t,x)) 
 &\le \int_0^t\int_0^1e^{2|b|_1 r}\sigma^2 G^2_{r}(x, z)\ud z \ud r=:\sigma^2_{\max}.
 \end{align*}
Therefore, from \eqref{rho0} we deduce immediately that $u(t,x)$ has a density $q_{t,x}$ satisfying, for almost all $\z\in\R$,
\begin{align}\nonumber
\frac{\E|u(t,x)-\E[u(t,x)]|}{2 \sigma_{\min }^{2}} &\exp \left(-\frac{(\z-\E[u(t,x)])^{2}}{2 \sigma_{\max }^{2}}\right) \\\label{rho}
&\le  q_{t,x}(\z) \le\frac{\E|u(t,x)-\E[u(t,x)]|}{2 \sigma_{\max }^{2}} \exp \left(-\frac{(\z-\E[u(t,x)])^{2}}{2 \sigma_{\min }^{2}}\right).
\end{align}
Notice that there exist some constants $C$ and $\widetilde C$ independent of $t$, $b$ and $\sigma$ such that 
\begin{align}\label{max-min}
Ct^{\frac{1}{2}}e^{-2t|b|_1}\le \sigma^2_{\min}\le\sigma^2_{\max}\le \widetilde Ct^{\frac{1}{2}}e^{2t|b|_1},\,\forall\,t\in(0,1],
\end{align}
in view of \eqref{GG0}.
We claim that
\begin{equation}\label{Gt2}
\lim _{t \rightarrow 0}t^{\frac{1}{2}}\int_{0}^{1}G^2_{t}(x,y) \ud y=\frac{1}{(1+\mathrm{sgn}(x(1-x)))\sqrt{ 2\pi}},\,\forall\,x\in[0,1].
\end{equation}
In fact, the spectral decomposition 
\begin{equation*}
G_{t}(x, y)=2 \sum_{k=0}^{\infty}e^{-k^{2} \pi^{2} t}\cos (k \pi x)\cos (k \pi y)
\end{equation*}
and the identity $2\cos(j\pi y)\cos (k \pi y)=\cos((j+k)\pi y)+\cos((j-k)\pi y)$
allow us to calculate
\begin{align*}
\int_0^1G^2_{t}(x, y)\ud y&=2 \sum_{k=0}^\infty\sum_{j=0}^\infty e^{(-k^2-j^2) \pi^{2}t}\cos (k\pi x)\cos(j\pi x)\int_0^12\cos(j\pi y)\cos (k \pi y)\ud y \\
&=2+2 \sum_{k=0}^\infty\sum_{j=0}^\infty e^{(-k^2-j^2) \pi^{2}t}\cos (k\pi x)\cos(j\pi x)\int_0^1\cos((j-k)\pi y)\ud y\\
&=2+2 \sum_{k=0}^\infty e^{-2k^2\pi^{2}t}\cos^2(k\pi x)=2+G_{2t}(x,x).
\end{align*}
Accordingly, for any $x\in[0,1]$, we have
\begin{align*}
&t^{\frac{1}{2}}\int_{0}^{1}G^2_{t}(x,y) \ud y=2t^{\frac{1}{2}}+t^{\frac{1}{2}}
G_{2t}(x,x)=2t^{\frac{1}{2}}+\frac{1}{\sqrt{8 \pi }}\sum_{n=-\infty}^{+\infty}\left(e^{-\frac{n^{2}}{2 t}}+e^{-\frac{(x- n)^{2}}{2t}}\right)\\
=&2t^{\frac{1}{2}}+\frac{1}{\sqrt{8 \pi }}\left\{1+e^{-\frac{(x-1)^{2}}{ 2t}}+e^{-\frac{x^{2}}{ 2t}}+2\sum_{n=1}^{\infty}e^{-\frac{n^{2}}{2t}}+\sum_{n=2}^{\infty}e^{-\frac{(x- n)^{2}}{2t}}+\sum_{n=1}^{\infty}e^{-\frac{(x+ n)^{2}}{2t}}\right\}.
\end{align*}
Using the Euler-Poisson integral: $\int_{\mathbb{R}} e^{-\pi x^{2}}\ud x=1$, we have
\begin{align}\label{pit}
0<\sum_{n=1}^{\infty}e^{-\frac{n^{2}}{ 2t}}\le\int_0^\infty e^{-\frac{y^{2}}{ 2t}}\ud y=\sqrt{\frac{\pi t}{2}}.
\end{align}
Combining \eqref{pit} and the fact $$0<\sum_{n=2}^{\infty}e^{-\frac{(x- n)^{2}}{2t}}+\sum_{n=1}^{\infty}e^{-\frac{(x+ n)^{2}}{2t}}\le 2\sum_{n=1}^{\infty}e^{-\frac{n^{2}}{2 t}},\,\forall\,x\in[0,1],$$ 
we arrive at
\begin{equation*}
\lim _{t \rightarrow 0}t^{\frac{1}{2}}\int_{0}^{1}G^2_{t}(x,y) \ud y=\left\{
\begin{split}
&\frac{1}{\sqrt{ 2\pi }},\quad \text{if $x\in\{0,1\}$},\\
&\frac{1}{2\sqrt{2 \pi }},\quad \text{if $x\in(0,1)$},
\end{split}\right.
\end{equation*}
which is equivalent to \eqref{Gt2}.

\textsf{Upper bound}: By \eqref{max-min}, \eqref{Gt2},  the right hand of \eqref{rho} and L'H\^opital's rule, we have
\begin{align*}
\lim _{t \rightarrow 0} t^{\frac{1}{2}} \log q_{t,x}(\z)&\le\lim _{t \rightarrow 0} t^{\frac{1}{2}}\log\frac{\E|u(t,x)-\E[u(t,x)]|}{2 \sigma_{\max }^{2}} +\lim _{t \rightarrow 0} t^{\frac{1}{2}} \left(-\frac{(\z-\E[u(t,x)])^{2}}{2 \sigma_{\min }^{2}}\right)\\
&=-(\z-u_0(x))^{2}\lim _{t \rightarrow 0}  \frac{t^{\frac{1}{2}}}{2 \sigma_{\min }^{2}}=-\frac{(\z-u_0(x))^{2}}{4}\lim _{t \rightarrow 0}  \frac{1}{e^{-2|b|_1 t}\sigma^2 t^{\frac{1}{2}}\int_0^1G^2_t(x, z)\ud z }\\
&=-\frac{(\z-u_0(x))^{2}}{4\sigma^2}(1+\mathrm{sgn}(x(1-x)))\sqrt{ 2\pi},
\end{align*}
where in the first step, we have used the facts that $\E[|u(t,x)|]$ is uniformly bounded with respect to $t\in(0,1]$ and $x\in[0,1]$ when dealing with the first limit and  that $\lim\limits_{t\rightarrow 0}\E[u(t,x)]=u(0,x)=u_0(x)$ when dealing with the second limit.

\textsf{Lower bound}: Similarly, from the left hand of \eqref{rho} and L'H\^opital's rule, it follows that
\begin{align*}
\lim _{t \rightarrow 0} t^{\frac{1}{2}} \log q_{t,x}(\z)&\ge-\frac{(\z-u_0(x))^{2}}{4\sigma^2}(1+\mathrm{sgn}(x(1-x)))\sqrt{2 \pi}.
\end{align*}
The proof is completed.

\end{proof}

From the above theorem, even if the drift term $b$ is nonlinear, the behavior of $q_{t,x}$ looks like a Gaussian density with mean $u_0(x)$ and covariance nearly proportional to $t^{\frac{1}{2}}$ when $t$ is sufficiently small, since
\begin{align*}
 q_{t,x}(\z)&\approx\exp\left(-\frac{\sqrt{ 2\pi}(1+\mathrm{sgn}(x(1-x)))(\z-u_0(x))^{2}}{4\sigma^2t^{\frac{1}{2}}}\right),\,0<t\ll1.
\end{align*}
Roughly speaking, for fixed $x\in[0,1]$,  the distribution of $u(t,x)$ decays to the distribution $\bm\delta_{u_0(x)}$ of $u(0,x)$ exponentially as $t$ tends to $0$.

Under Dirichlet boundary condition, we denote by $\mathbf u(t,x)$ the corresponding solution to Eq. \eqref{SHE}. 
By a slight modification, a similar result can be proved.

\begin{cor}\label{Dir-Cor}
Under the same condition of Theorem \ref{Asy}, except replacing the Neumann boundary condition by the Dirichlet boundary condition,  for any $x\in(0,1)$,  $\mathbf u(t,x)$ 
 admits a density $\mathbf{q}_{t,x}$ satisfying  that for almost every $\z\in\R$,
\begin{align*}
\lim _{t \rightarrow 0} t^{\frac{1}{2}} \log\mathbf q_{t,x}(\z)=-\frac{\sqrt{ 2\pi}}{2\sigma^2}(\z-u_0(x))^{2}.
\end{align*}
\end{cor}


 We also investigate the logarithmic asymptotic property of the density of the approximation $\{u^\delta(\delta,x)\}_{\delta>0}$ associated to Eq. \eqref{SAEE}, as the perturbation parameter $\delta$ tends to $0$. 
 It is
observed that the limit $\lim _{\delta  \rightarrow 0} \delta ^{\frac{1}{2}} \log q^\delta_{\delta ,x}(\z)$ is exactly the limit $\lim _{t \rightarrow 0} t^{\frac{1}{2}} \log q_{t,x}(\z)$.   \begin{pro}\label{Ndensity}
Assume that $b\in\mathcal C_{\mathbf b}^1$.  Then for any $x\in[0,1]$, the solution $u^\delta(\delta,x)$ given by  Eq. \eqref{SAEE} admits a density $q^\delta_{\delta ,x}$ satisfying  that for almost every $\z\in\R$,
\begin{align*}
\lim _{\delta  \rightarrow 0} \delta ^{\frac{1}{2}} \log q^\delta_{\delta,x}(\z)&=-\frac{\sqrt{ 2\pi}}{4\sigma^2}(1+\mathrm{sgn}(x(1-x)))(\z-u_0(x))^{2}.
\end{align*}
\end{pro}

The proofs of Corollary \ref{Dir-Cor} and Proposition \ref{Ndensity} are similar to that of Theorem \ref{Asy} and are postponed to Appendix.

\textbf{Acknowledgments.}
The authors are very grateful to Charles-Edouard Br\'ehier (Univ Lyon)  for his helpful discussions and suggestions.

\section*{Appendix}\label{Appendix}

In the Appendix, we give the proofs of some technique results for reader's convenience.

\subsection*{Proof of Lemma \ref{GW}}
\begin{proof}
Taking the supremum over $x\in(0,1)$, then for any $y\in(0,1)$,
\begin{equation*}
\sup_{x\in(0,1)}g_{s,y}(t,x)\le \frac{C}{\sqrt{t-s}}+C\int_{s}^{t}\sup_{z_1\in(0,1)}g_{s,y}(r_1,z_1)\ud r_1,
\end{equation*}
which, together with Gronwall's inequality, implies that for some $C=C(T)$,
\begin{equation}\label{gsy}
\sup_{x\in(0,1)}g_{s,y}(t,x)\le \frac{C}{\sqrt{t-s}}+C,\forall\,y\in(0,1).
\end{equation}
By an iteration process and the semigroup property of $G$, we have
\begin{align*}
 &g_{s, y}(t, x)\le C G_{t-s}(x,y)+C \int_{s}^{t} \int_0^1 G_{t-r_{1}}(x,z_{1})G_{r_{1}-s}(z_{1},y)\ud z_1\ud r_1+\cdots\\
 &\quad+C^{n} \int_{s}^{t}\int_0^1\int_s^{r_1}\int_0^1\cdots\int_s^{r_{n-1}}\int_0^1G_{t-r_1}(x,z_1)G_{r_1-r_2}(z_1,z_2)\\
 &\qquad\cdots G_{r_{n-1}-r_n}(z_{n-1},z_n)G_{r_n-s}(z_n,y)\ud z_n\ud r_n\cdots\ud z_2\ud r_2\ud z_1\ud r_1\\
 &\quad+C^n\int_{s}^{t}\int_0^1\int_s^{r_1}\int_0^1\cdots\int_s^{r_n}\int_0^1G_{t-r_1}(x,z_1)G_{r_1-r_2}(z_1,z_2)\\
 &\quad\cdots G_{r_{n-1}-r_n}(z_{n-1},z_n)G_{r_n-r_{n+1}}(z_n,z_{n+1})g_{s, y}(r_{n+1},z_{n+1})\ud z_{n+1}\ud r_{n+1}\cdots\ud z_2\ud r_2\ud z_1\ud r_1\\
 &\le \left(C+ C(t-s)+\cdots+C^n\frac{(t-s)^n}{n!}\right)G_{t-s}(x,y)\\
 &\quad+C^n\frac{(t-s)^n}{n!}\int_s^t\int_0^1G_{t-r_{n+1}}(x,z_{n+1})g_{s, y}(r_{n+1},z_{n+1})\ud z_{n+1}\ud r_{n+1},
  \end{align*}
where the first term on the right-hand side is bounded by $e^{CT}G_{t-s}(x,y)$ and the second term 
 is dominated by 
$$C^n\frac{(t-s)^n}{n!}C(T)\left(\int_s^t\frac{1}{\sqrt{t-r_{n+1}}}\frac{1}{\sqrt{r_{n+1}-s}}\ud r_{n+1}+1\right),$$  which tends to $0$ as $n\rightarrow \infty$, thanks to \eqref{gsy}. The proof is completed.
\end{proof}

\subsection*{Proof of Corollary \ref{Dir-Cor}}
\begin{proof}
In contrast to \eqref{Gt2},
it suffices to show that for every $x\in(0,1),$
\begin{equation}\label{Gt21}
\lim _{t \rightarrow 0}t^{\frac{1}{2}}\int_{0}^{1}\mathbf G^2_{t}(x,y) \ud y=\frac{1}{2\sqrt{ 2\pi}}.
\end{equation} 
Fix $x\in(0,1)$. Now we proceed to verify \eqref{Gt21}.
 Indeed, the expression 
\begin{equation*}
\mathbf G_{t}(x, y)=2 \sum_{k=1}^{\infty}e^{-k^{2} \pi^{2} t}\sin (k \pi x)\sin (k \pi y)
\end{equation*}
and the identity $2\sin(j\pi y)\sin (k \pi y)=-\cos((j+k)\pi y)+\cos((j-k)\pi y)$ yield
\begin{align*}
\int_0^1\mathbf G^2_{t}(x, y)\ud y&=\mathbf G_{2t}(x,x).
\end{align*}
Furthermore, by applying \eqref{Diri}, we have
\begin{align*}
&t^{\frac{1}{2}}\int_{0}^{1}\mathbf G^2_{t}(x,y) \ud y\\
=&
\frac{1}{\sqrt{8 \pi }}\left\{1-e^{-\frac{(x-1)^{2}}{ 2t}}-e^{-\frac{x^{2}}{ 2t}}+\left(\sum_{n=1}^{\infty}e^{-\frac{n^{2}}{2t}}-\sum_{n=2}^{\infty}e^{-\frac{(x- n)^{2}}{2t}}\right)+\left(\sum_{n=1}^{\infty}e^{-\frac{n^{2}}{2t}}-\sum_{n=1}^{\infty}e^{-\frac{(x+ n)^{2}}{2t}}\right)\right\}
\end{align*}
with 
$$0\le\sum_{n=1}^{\infty}e^{-\frac{n^{2}}{2t}}-\sum_{n=2}^{\infty}e^{-\frac{(x- n)^{2}}{2t}}\le \sqrt{\frac{\pi t}{2}}$$
and
$$0\le\sum_{n=1}^{\infty}e^{-\frac{n^{2}}{2t}}-\sum_{n=1}^{\infty}e^{-\frac{(x+ n)^{2}}{2t}}\le \sqrt{\frac{\pi t}{2}}.$$
Putting $t\rightarrow 0$, the desired identity \eqref{Gt21} follows and the proof is finished.
\end{proof}
\subsection*{Proof of Corollary \ref{Ndensity}}
 \begin{proof}
First, we fix $x\in(0,1)$.
By \eqref{perturbed sol}, $u^\delta(t,x )|_{t=\delta}$ is computed by
\begin{align*}
u^\delta(\delta,x )=&\int_{0}^{1}G_{\delta}(x,y)u_0(y)\ud y+\int^{\delta}_{0}\int_{0}^{1}G_{\delta-s}(x,y)b\left(u_0(y)\right)\ud y\ud s \\
&+\int_{0}^{\delta}\int_{0}^{1}G_{\delta-s}(x,y)\sigma W(\ud s,\ud y),
\end{align*}
which implies that the distribution of $u^\delta(t,x )|_{t=\delta}$ is Gaussian and hence is denoted by $\mathcal N(\mu_{\delta},\nu_{\delta})$. By applying the isometry formula, we have
\begin{align*}
&\mu_{\delta}=\int_{0}^{1}G_{\delta}(x,y)u_0(y)\ud y+\int_{0}^{\delta}\int_{0}^{1}G_{\delta-s}(x,y)b\left(u_0(y)\right)\ud y\ud s ,\\
&\nu_{\delta}=\int_{0}^{\delta}\int_{0}^{1}G^2_{\delta-s}(x,y)\sigma^2\ud y\ud s=\sigma^2\int_{0}^{\delta}\int_{0}^{1}G^2_{s}(x,y)\ud y\ud s.
\end{align*}
Therefore,
\begin{align}\nonumber\label{limita}
\lim _{\delta  \rightarrow 0} \delta^{\frac{1}{2}} \log q^\delta_{\delta,x }(\z)&=\lim _{\delta  \rightarrow 0} \delta^{\frac{1}{2}} \log \frac{1}{\sqrt{2\pi\nu_{\delta}}}e^{-\frac{(\z-\mu_{\delta})^2}{2\nu_{\delta}}}\\
=&\lim _{\delta  \rightarrow 0} \delta^{\frac{1}{2}} \log \frac{1}{\sqrt{2\pi\nu_{\delta}}}+\lim _{\delta  \rightarrow 0}  -\delta^{\frac{1}{2}}\frac{(\z-\mu_{\delta})^2}{2\nu_{\delta}}.
\end{align}
 Taking $|b|_1=0$ in \eqref{max-min} yields
$C\delta^{\frac{1}{2}}\le \int_{0}^{\delta}\int_{0}^{1}G^2_{s}(x,y)\ud y\ud s\le \widetilde C\delta^{\frac{1}{2}},\,\forall\,\delta\in(0,1],$ which indicates the first limit in the right hand of \eqref{limita} is zero. Observing that $\lim _{\delta  \rightarrow 0}\mu_{\delta}=u_0(x)$ and using \eqref{Gt2}, as well as L'H\^opital's rule, we finally derive that
\begin{align*}
\lim _{\delta  \rightarrow 0} \delta^{\frac{1}{2}} \log q^\delta_{\delta,x }(\z)=
\lim _{\delta  \rightarrow 0}  -\delta^{\frac{1}{2}}\frac{(\z-u_0(x))^2}{2\nu_{\delta}}=-\frac{\sqrt{ 2\pi}}{4\sigma^2}(1+\mathrm{sgn}(x(1-x)))(\z-u_0(x))^{2}.
\end{align*}
\end{proof}



\subsection*{Proof of Proposition \ref{DZ}}
\begin{proof}
Similar to the proof of \eqref{Yiy1}, for $\theta\in(0,t_{i-1})$, 
\begin{align*}
\|D_{\theta,\xi}Z_i^\beta(r,y)\|_{k,p}
&\le C G_{r-\theta}(y,\xi).
\end{align*}
Similar to the proof of Lemma \ref{Ykp0},
for $\theta\in(t_{i-1},r)$, we only give the details of the case $k=0$, and the induction argument for $k\ge1$ is omitted.
By the definition of $Z_i^\beta(r,y)$, for $\theta\in(t_{i-1},r)$, 
\begin{align*}
D_{\theta,\xi}&Z_i^\beta(r,y)=\beta D_{\theta,\xi}\varphi_{r}^y(t_{i-1},\Phi_{t_{i-1}}(0,u_0))
= \beta\sigma G_{r-\theta}(y,\xi)\\
&+\beta\int_{\theta}^{r}\int_0^1G_{r-s}(y,z)b^\prime(\varphi_s^z(t_{i-1},\Phi_{t_{i-1}}(0,u_0)))D_{\theta,\xi}\varphi_s^z(t_{i-1},\Phi_{t_{i-1}}(0,u_0))\ud z\ud s,
\end{align*}
which together with Lemma \ref{GW} gives 
\begin{align*}
\|D_{\theta,\xi}Z_i^\beta(r,y)\|_{p}\le C G_{r-\theta}(y,\xi).
\end{align*}

\end{proof}

\bibliographystyle{plain}
\bibliography{pdfreference}

\end{document}